\documentclass{article}

\usepackage{arxiv}

\usepackage[utf8]{inputenc}
\usepackage[T1]{fontenc}
\usepackage{amsmath}
\usepackage{amsfonts}
\usepackage{amssymb}
\usepackage{amsthm}
\usepackage{mathtools}
\usepackage[colorlinks=true, linkcolor=blue]{hyperref}
\usepackage{cleveref}
\usepackage{appendix}
\usepackage[inline]{enumitem}

\newtheorem{theorem}{Theorem}[section]

\newtheorem{lemma}[theorem]{Lemma}
\newtheorem{corollary}[theorem]{Corollary}

\newtheorem{remark}[theorem]{Remark}
\newtheorem{definition}[theorem]{Definition}

\numberwithin{equation}{section}

\allowdisplaybreaks

\usepackage{amsmath}
\usepackage{amsfonts}
\usepackage{amssymb}
\usepackage{mathtools}
\usepackage{tikz}
\usepackage{pgfplots}
\usepackage{dsfont}

\usepackage{acronym}

\acrodef{fom}[FOM]{full-order model}
\acrodef{rom}[ROM]{reduced-order model}
\acrodef{strom}[StROM]{structured \ac{rom}}
\acrodef{diagstrom}[D-StROM]{diagonal \ac{strom}}

\acrodef{siso}[SISO]{single-input single-output}
\acrodef{mimo}[MIMO]{multiple-input multiple-output}

\acrodef{lti}[LTI]{linear time-invariant}
\acrodefindefinite{lti}{an}{a}

\acrodef{ph}[pH]{port-Hamiltonian}

\acrodef{irka}[IRKA]{iterative rational Krylov algorithm}

\newcommand{\nfom}{\ensuremath{n}}
\newcommand{\nin}{\ensuremath{n_{\textnormal{i}}}}
\newcommand{\nout}{\ensuremath{n_{\textnormal{o}}}}

\newcommand{\yf}{\ensuremath{H}}

\newcommand{\nrom}{\ensuremath{r}}

\newcommand{\Ar}{\ensuremath{\hat{\mathcal{A}}}}
\newcommand{\Br}{\ensuremath{\hat{\mathcal{B}}}}
\newcommand{\Cr}{\ensuremath{\hat{\mathcal{C}}}}

\newcommand{\cAr}{\ensuremath{\hat{A}}}
\newcommand{\cBr}{\ensuremath{\hat{B}}}
\newcommand{\cCr}{\ensuremath{\hat{C}}}

\newcommand{\car}{\ensuremath{\hat{\alpha}}}
\newcommand{\cbr}{\ensuremath{\hat{\beta}}}
\newcommand{\ccr}{\ensuremath{\hat{\gamma}}}

\newcommand{\qAr}{\ensuremath{q_{\Ar}}}
\newcommand{\qBr}{\ensuremath{q_{\Br}}}
\newcommand{\qCr}{\ensuremath{q_{\Cr}}}

\newcommand{\xr}{\ensuremath{\hat{X}}}
\newcommand{\yr}{\ensuremath{\hat{H}}}

\newcommand{\xrd}{\ensuremath{\xr_{d}}}

\newcommand{\hA}{\ensuremath{\hat{A}}}
\newcommand{\hB}{\ensuremath{\hat{B}}}
\newcommand{\hC}{\ensuremath{\hat{C}}}
\newcommand{\hE}{\ensuremath{\hat{E}}}
\newcommand{\hG}{\ensuremath{\hat{G}}}
\newcommand{\hH}{\ensuremath{\hat{H}}}
\newcommand{\hI}{\ensuremath{\hat{I}}}
\newcommand{\hJ}{\ensuremath{\hat{J}}}
\newcommand{\hK}{\ensuremath{\hat{K}}}
\newcommand{\hM}{\ensuremath{\hat{M}}}

\newcommand{\hR}{\ensuremath{\hat{R}}}
\newcommand{\hS}{\ensuremath{\hat{S}}}
\newcommand{\hT}{\ensuremath{\hat{T}}}
\newcommand{\hX}{\ensuremath{\hat{X}}}
\newcommand{\hY}{\ensuremath{\hat{Y}}}

\newcommand{\hx}{\ensuremath{\hat{x}}}
\newcommand{\hy}{\ensuremath{\hat{y}}}

\newcommand{\hSigma}{\ensuremath{\hat{\Sigma}}}

\newcommand{\obj}{\ensuremath{\mathcal{J}}}

\newcommand{\cH}{\ensuremath{\mathcal{H}}}
\newcommand{\cL}{\ensuremath{\mathcal{L}}}
\newcommand{\Htwo}{\ensuremath{\cH_{2}}}
\newcommand{\htwo}{\ensuremath{h_{2}}}
\newcommand{\Ltwo}{\ensuremath{\cL_{2}}}
\newcommand{\Hinf}{\ensuremath{\cH_{\infty}}}
\newcommand{\Linf}{\ensuremath{\cL_{\infty}}}

\newcommand{\HtwoLtwo}{\ensuremath{\cH_{2} \otimes \cL_{2}}}

\newcommand{\pp}{\ensuremath{\mathsf{p}}}
\newcommand{\pset}{\ensuremath{\mathcal{P}}}
\newcommand{\npar}{\ensuremath{n_{\pp}}}

\newcommand{\CC}{\ensuremath{\mathbb{C}}}

\newcommand{\RR}{\ensuremath{\mathbb{R}}}

\newcommand{\CCpar}{\ensuremath{\CC^{\npar}}}

\newcommand{\CCi}{\ensuremath{\CC^{\nin}}}
\newcommand{\CCo}{\ensuremath{\CC^{\nout}}}

\newcommand{\CCoi}{\ensuremath{\CC^{\nout \times \nin}}}

\newcommand{\CCor}{\ensuremath{\CC^{\nout \times \nrom}}}
\newcommand{\CCri}{\ensuremath{\CC^{\nrom \times \nin}}}
\newcommand{\CCro}{\ensuremath{\CC^{\nrom \times \nout}}}
\newcommand{\CCrr}{\ensuremath{\CC^{\nrom \times \nrom}}}

\newcommand{\RRf}{\ensuremath{\RR^{\nfom}}}
\newcommand{\RRi}{\ensuremath{\RR^{\nin}}}
\newcommand{\RRo}{\ensuremath{\RR^{\nout}}}

\newcommand{\RRff}{\ensuremath{\RR^{\nfom \times \nfom}}}
\newcommand{\RRfi}{\ensuremath{\RR^{\nfom \times \nin}}}
\newcommand{\RRof}{\ensuremath{\RR^{\nout \times \nfom}}}

\newcommand{\RRor}{\ensuremath{\RR^{\nout \times \nrom}}}
\newcommand{\RRri}{\ensuremath{\RR^{\nrom \times \nin}}}

\newcommand{\RRrr}{\ensuremath{\RR^{\nrom \times \nrom}}}

\DeclarePairedDelimiter{\myparen}{\lparen}{\rparen}

\DeclarePairedDelimiter{\abs}{\lvert}{\rvert}
\DeclarePairedDelimiter{\norm}{\lVert}{\rVert}
\DeclarePairedDelimiterXPP{\normtwo}[1]{}{\lVert}{\rVert}{_{2}}{#1}
\DeclarePairedDelimiterXPP{\normF}[1]{}{\lVert}{\rVert}{_{\operatorname{F}}}{#1}
\DeclarePairedDelimiterXPP{\normHinf}[1]{}{\lVert}{\rVert}{_{\Hinf}}{#1}
\DeclarePairedDelimiterXPP{\normHtwoLtwo}[1]{}{\lVert}{\rVert}{_{\HtwoLtwo}}{#1}
\DeclarePairedDelimiterXPP{\normHtwo}[1]{}{\lVert}{\rVert}{_{\Htwo}}{#1}
\DeclarePairedDelimiterXPP{\normLinfLtwo}[1]{}{\lVert}{\rVert}{_{\Linf \otimes \Ltwo}}{#1}
\DeclarePairedDelimiterXPP{\normLinfmu}[1]{}{\lVert}{\rVert}{_{\Linf(\pset, \measure)}}{#1}
\DeclarePairedDelimiterXPP{\normLinf}[1]{}{\lVert}{\rVert}{_{\Linf}}{#1}
\DeclarePairedDelimiterXPP{\normLtwomu}[1]{}{\lVert}{\rVert}{_{\Ltwo(\pset, \measure)}}{#1}
\DeclarePairedDelimiterXPP{\normLtwo}[1]{}{\lVert}{\rVert}{_{\Ltwo}}{#1}
\DeclarePairedDelimiterXPP{\normhtwo}[1]{}{\lVert}{\rVert}{_{\htwo}}{#1}
\DeclarePairedDelimiterXPP{\ip}[2]{}{\langle}{\rangle}{}{#1, #2}
\DeclarePairedDelimiterXPP{\ipF}[2]{}{\langle}{\rangle}{_{\operatorname{F}}}{#1, #2}
\DeclarePairedDelimiterXPP{\ipHtwo}[2]{}{\langle}{\rangle}{_{\Htwo}}{#1, #2}
\DeclarePairedDelimiterXPP{\iphtwo}[2]{}{\langle}{\rangle}{_{\htwo}}{#1, #2}
\DeclarePairedDelimiterXPP{\ipLtwo}[2]{}{\langle}{\rangle}{_{\Ltwo}}{#1, #2}
\DeclarePairedDelimiterXPP{\ipLtwomu}[2]{}{\langle}{\rangle}{_{\Ltwo(\pset, \measure)}}{#1, #2}
\DeclarePairedDelimiterXPP{\mydiag}[1]{\operatorname{diag}}{\lparen}{\rparen}{}{#1}
\DeclarePairedDelimiterXPP{\trace}[1]{\operatorname{tr}}{\lparen}{\rparen}{}{#1}
\DeclarePairedDelimiterXPP{\vecop}[1]{\operatorname{vec}}{\lparen}{\rparen}{}{#1}
\DeclarePairedDelimiterXPP{\Real}[1]{\operatorname{Re}}{\lparen}{\rparen}{}{#1}
\DeclarePairedDelimiterXPP{\dist}[1]{\operatorname{dist}}{\lparen}{\rparen}{}{#1}
\DeclarePairedDelimiterXPP{\len}[1]{\operatorname{length}}{\lparen}{\rparen}{}{#1}
\DeclarePairedDelimiterXPP{\myspan}[1]{\operatorname{span}}{\lbrace}{\rbrace}{}{#1}

\newcommand{\measure}{\ensuremath{\mu}}

\DeclareMathOperator{\dif}{d\!}

\DeclareMathOperator*{\esssup}{ess\,sup}

\DeclareMathOperator{\Res}{Res}
\DeclareMathOperator{\Proj}{Proj}

\newcommand{\difm}[1]{\ensuremath{\dif{\measure(#1)}}}

\newcommand{\fundef}[3]{\ensuremath{#1 \colon #2 \to #3}}

\newcommand{\tran}{^{\operatorname{T}}}

\newcommand{\herm}{^{*}}
\newcommand{\mherm}{^{-*}}

\newcommand{\imag}{\boldsymbol{\imath}}

\definecolor{mplC0}{HTML}{1F77B4}
\definecolor{mplC1}{HTML}{FF7F0E}
\definecolor{mplC2}{HTML}{2CA02C}
\definecolor{mplC3}{HTML}{D62728}
\definecolor{mplC4}{HTML}{9467BD}
\definecolor{mplC5}{HTML}{8C564B}
\definecolor{mplC6}{HTML}{E377C2}
\definecolor{mplC7}{HTML}{7F7F7F}
\definecolor{mplC8}{HTML}{BCBD22}
\definecolor{mplC9}{HTML}{17BECF}

\pgfplotsset{compat=1.10}
\pgfplotscreateplotcyclelist{mpl}{
  mplC0, very thick \\
  mplC1, very thick, densely dashed \\
  mplC2, very thick, densely dotted \\
  mplC3, very thick, densely dashdotted \\
}
\pgfplotscreateplotcyclelist{mplmark}{
  mplC0, very thick, mark=o, every mark/.append style={solid} \\
  mplC1, very thick, densely dashed, mark=x, every mark/.append style={solid}, mark size=3 \\
  mplC2, very thick, densely dotted, mark=square, every mark/.append style={solid} \\
  mplC3, very thick, densely dashdotted, mark=diamond, every mark/.append style={solid} \\
}
\pgfplotscreateplotcyclelist{mplmarkonly}{
  draw=mplC0, fill=mplC0, only marks, mark=*, mark size=1 \\
  draw=mplC1, fill=mplC1, only marks, mark=square*, mark size=1 \\
  draw=mplC2, fill=mplC2, only marks, mark=triangle*, mark size=1.5 \\
}

\let\hat\widehat%
\let\le\leqslant%
\let\ge\geqslant%

\begin{document}

\title{Interpolatory \texorpdfstring{$\Htwo$}{H2}-optimality Conditions for
  Structured Linear Time-invariant Systems\thanks{%
    This work was partially funded by the U.S. National Science Foundation under
    grant DMS-1923221.
  }}

\author{%
  \href{https://orcid.org/0000-0002-9437-7698}{%
    \includegraphics[scale=0.06]{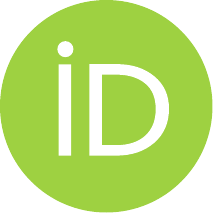}}\hspace{1mm}Petar~Mlinari\'c%
  \thanks{Department of Mathematics, Virginia Tech, Blacksburg, VA 24061
    (\texttt{mlinaric@vt.edu}).}
	\And
  \href{https://orcid.org/0000-0003-3362-4103}{%
    \includegraphics[scale=0.06]{orcid.pdf}}\hspace{1mm}Peter~Benner%
  \thanks{Computational Methods in Systems and Control Theory,
    Max Planck Institute for Dynamics of Complex Technical Systems,
    39106 Magdeburg,
    Germany
    (\texttt{benner@mpi-magdeburg.mpg.de}).}
	\And
	\href{https://orcid.org/0000-0003-4564-5999}{%
    \includegraphics[scale=0.06]{orcid.pdf}}\hspace{1mm}Serkan~Gugercin%
  \thanks{Department of Mathematics and Division of Computational Modeling and
    Data Analytics, Academy of Data Science, Virginia Tech, Blacksburg, VA 24061
    (\texttt{gugercin@vt.edu}).}
}

\renewcommand{\headeright}{Preprint}
\renewcommand{\undertitle}{Preprint}
\renewcommand{\shorttitle}{Interpolatory Structured $\Htwo$-optimality
  Conditions}

\hypersetup{
  pdftitle={Interpolatory H2-optimality Conditions for Structured Linear
    Time-invariant Systems},
  pdfauthor={P. Mlinari\'c, P. Benner, and S. Gugercin}
}

\maketitle

\begin{abstract}
  Interpolatory necessary optimality conditions for
$\Htwo$-optimal reduced-order modeling of
unstructured linear time-invariant (LTI) systems are well-known.
Based on previous work on $\Ltwo$-optimal reduced-order modeling of
stationary parametric problems, in this paper
we develop and investigate optimality conditions for
$\Htwo$-optimal reduced-order modeling of
structured LTI systems, in particular, for
second-order, port-Hamiltonian, and time-delay systems.
Under certain diagonalizability assumptions,
we show that across all these different structured settings,
bitangential Hermite interpolation is the common form for optimality,
thus proving a unifying optimality framework for
structured reduced-order modeling.

\end{abstract}


\section{Introduction}
\Ac{lti} dynamical systems are ubiquitous in many applications and
can often be represented in a finite-dimensional, unstructured state-space form
\begin{subequations}\label{eq:lti-fom}
  \begin{align}
    \dot{x}(t) & = A x(t) + B u(t), \\*
    y(t)       & = C x(t),
  \end{align}
\end{subequations}
where
$A \in \RRff$,
$B \in \RRfi$, and
$C \in \RRof$;
$x(t) \in \RRf$ are the states (internal degrees of freedom) of the system,
$u(t) \in \RRi$ are the inputs (forcing), and
$y(t) \in \RRo$ are the outputs (quantities of interest).
Given the \ac{lti} system~\eqref{eq:lti-fom}
(called \iac{fom}),
the goal of (unstructured) $\Htwo$-optimal reduced-order modeling is to find
\iac{rom} of order~$\nrom \ll \nfom$ given in the state-space form as
\begin{subequations}\label{eq:lti-rom}
  \begin{align}
    \dot{\hx}(t) & = \hA \hx(t) + \hB u(t), \\*
    \hy(t)       & = \hC \hx(t),
  \end{align}
\end{subequations}
with
$\hA \in \RRrr$,
$\hB \in \RRri$, and
$\hC \in \RRor$
such that the \ac{rom}~\eqref{eq:lti-rom} minimizes the squared $\Htwo$ error
\begin{equation}\label{eq:h2-error}
  \normHtwo*{H - \hH}^2
  = \frac{1}{2 \pi}
  \int_{-\infty}^{\infty}
  \normF*{H(\imag \omega) - \hH(\imag \omega)}^2
  \dif{\omega},
\end{equation}
where $H$ and $\hH$ are \emph{the transfer functions} of
the \ac{fom} and the \ac{rom},
respectively, and are given by
\begin{equation*}
  H(s) = C \myparen{s I - A}^{-1} B
  \ \text{ and } \
  \hH(s) = \hC \myparen*{s \hI - \hA}^{-1} \hB,
\end{equation*}
with $I \in \RRff$ and $\hI \in \RRrr$ being identity matrices.
In particular, transfer functions satisfy
\begin{equation*}
  Y(s) = H(s) U(s)
  \ \text{ and } \
  \hY(s) = \hH(s) U(s),
\end{equation*}
where $U, Y, \hY$ are Laplace transforms of the input and output signals
$u, y, \hy$.
Additional assumptions are that $A$ and $\hA$ have eigenvalues in $\CC_-$
(the open left half-plane),
ensuring that $H$ and $\hH$ have finite $\Htwo$ norms and
are analytic over $\CC_+$ (the open right half-plane).
Finding \iac{rom} that minimizes the $\Htwo$ error is motivated by the $\Linf$
output error bound, namely
\begin{equation*}
  \normLinf*{y - \hy}
  \le
  \normHtwo*{H - \hH}
  \normLtwo*{u}.
\end{equation*}
Thus, a small $\Htwo$ error guarantees a small $\Linf$ output error
(in the time-domain).
In the following, the function $H$ need not be a transfer function of
a finite-dimensional \ac{lti} system as in~\eqref{eq:lti-fom},
but it can be any element of the Hardy $\Htwo^{\nout \times \nin}$ space, i.e.,
a function $\fundef{H}{\CC}{\CCoi}$ analytic over $\CC_+$ with a finite $\Htwo$
norm~\cite{BeaG12}.

Assuming that $\hA$ has $\nrom$ distinct eigenvalues,
we can write the transfer function $\hH$ in \emph{the pole-residue form}
\begin{equation}\label{eq:pole-res-unstruct}
  \hH(s)
  = \hC \myparen*{s \hI - \hA}^{-1} \hB
  = \sum_{i = 1}^{\nrom} \frac{c_i b_i\herm}{s - \lambda_i}
\end{equation}
with pairwise distinct poles $\lambda_i$
(here $(\cdot)\herm$ denotes the complex conjugate).
Then the well-known necessary $\Htwo$-optimality conditions
in interpolation form~\cite{MeiL67,GugAB08}
state that an $\Htwo$-optimal \ac{rom}
with the transfer function in the pole-residue form~\eqref{eq:pole-res-unstruct}
needs to satisfy
\begin{subequations}\label{eq:lti-oc}
  \begin{align}
    \label{eq:lti-oc1}
    H\myparen*{-\overline{\lambda_i}}
    b_i
     & =
    \hH\myparen*{-\overline{\lambda_i}}
    b_i,                                 \\*
    \label{eq:lti-oc2}
    c_i\herm
    H\myparen*{-\overline{\lambda_i}}
     & =
    c_i\herm
    \hH\myparen*{-\overline{\lambda_i}}, \\*
    \label{eq:lti-oc3}
    c_i\herm
    H'\myparen*{-\overline{\lambda_i}}
    b_i
     & =
    c_i\herm
    \hH'\myparen*{-\overline{\lambda_i}}
    b_i,
  \end{align}
\end{subequations}
for $i = 1, 2, \ldots, \nrom$.
The necessary conditions~\eqref{eq:lti-oc1} and~\eqref{eq:lti-oc2} are,
respectively, right and left tangential Lagrange interpolation conditions,
and~\eqref{eq:lti-oc3} are bitangential Hermite conditions.
These conditions simply state that
the $\Htwo$-optimal rational approximant $\hH$
needs to be a bitangential Hermite interpolant to $H$
at the mirror of the poles of $\hH$
along the tangential directions resulting from the residues of $\hH$.
In other words,
the poles and residues of $\hH$ specify the optimal interpolation points and
the tangent directions.
The optimality conditions above assume simple poles,
our main focus in this paper.
For extension to poles with multiplicities, see~\cite{VanGA10}.

Interpolation conditions~\eqref{eq:lti-oc} can be enforced using
either a Petrov-Galerkin projection (intrusively) or
via data-driven Loewner matrices (non-intrusively).
This observation has led to \ac{irka}~\cite{GugAB08}
(projection-based formulation) and
transfer-function \ac{irka} (TF-IRKA)~\cite{BeaG12},
an efficient framework for locally $\Htwo$-optimal reduced-order modeling.
For more details on interpolatory model reduction,
we refer the reader to~\cite{AntBG20}.

Structure-preserving reduced-order modeling methods inspired by the unstructured
interpolatory conditions~\eqref{eq:lti-oc} were proposed, such as
SO-IRKA for second-order systems~\cite{Wya12},
IRKA-PH for port-Hamiltonian systems~\cite{GugPB+12}, and
SPTF-IRKA more generally~\cite{SinGB16}.
They yield high-fidelity \acp{rom}, but the question remains whether they can
achieve $\Htwo$-optimality or if there are similar methods that can achieve it.

The conditions~\eqref{eq:lti-oc} follow due to the (unstructured) pole-residue
form~\eqref{eq:pole-res-unstruct} of the transfer function $\hH$.
In this paper, we want to understand what happens to the interpolatory
conditions~\eqref{eq:lti-oc}
when the \ac{rom} is a structured system, e.g.,
a second-order system or a time-delay system, and
the pole-residue form needs to be modified accordingly.
As the optimization set changes,
it is expected that the necessary optimality conditions also change.
The question then is,
whether they are still in the form of bitangential Hermite conditions.

Recent work~\cite{MliG23a,MliG23b} introduced a unifying $\Ltwo$-optimal
reduced-order modeling framework,
covering both \ac{lti} systems and stationary parametric problems.
In particular,~\cite{MliG23b} showed that necessary optimality conditions
in the form of bitangential Hermite interpolation appear often,
demonstrated by recovering known conditions for (parametric) \ac{lti} systems
and deriving new ones for discretized $\Htwo$-optimal reduced-order modeling and
for a class of stationary parametric problems.
Here, we are interested in certain important classes of structured \ac{lti}
systems:
second-order systems,
\ac{ph} systems, and
time-delay systems.

Our main contributions are the following:
\begin{enumerate}
  \item Extending the work of~\cite{MliG23a,MliG23b} to complex-valued models
        (\Cref{sec:l2}).
  \item Developing interpolatory necessary $\Htwo$-optimality conditions for a
        class of diagonal structured systems
        (\Cref{sec:h2-diag}).
  \item Deriving interpolatory optimality conditions for specific diagonal
        structured \ac{lti} systems:
        (a) second-order systems (\Cref{{sec:so}}),
        (b) \ac{ph} systems (\Cref{sec:ph}), and
        (c) time-delay systems (\Cref{sec:td}).
  \item In almost all cases,
        we show that bitangential Hermite interpolation is the key uniform
        framework behind optimality.
\end{enumerate}

The rest of the paper is organized as follows.
In \Cref{sec:l2},
we recall some of the main results from~\cite{MliG23a,MliG23b}
on $\Ltwo$-optimal reduced-order modeling and
extend them to \acp{rom} with complex and/or diagonal matrices.
In \Cref{sec:h2-diag},
we specialize this discussion to $\Htwo$-optimal reduced-order modeling of
diagonal dynamical systems, and
then use this result in \Cref{sec:so,sec:ph,sec:td} to respectively cover
second-order, \ac{ph}, and time-delay systems.
We conclude with \Cref{sec:conclusion}.

\section{\texorpdfstring{$\Ltwo$}{L2}-optimal Reduced-order Modeling}%
\label{sec:l2}
Here we recall the setting of~\cite{MliG23a} and
extend it to \acp{strom} with complex matrices.
Then we generalize
the gradients of the squared $\Ltwo$ error from~\cite{MliG23a} and
necessary $\Ltwo$-optimality conditions from~\cite{MliG23b}
to complex \acp{strom}.
The final section is on the extension to \acp{diagstrom}.

\subsection{Setting}
Note that the transfer function of the \ac{rom}~\eqref{eq:lti-rom}
can be reformulated as the output of a parametric stationary problem
\begin{align*}
  \myparen*{s \hI - \hA} \hX(s) & = \hB,        \\*
  \hH(s)                        & = \hC \hX(s),
\end{align*}
where $s \in \imag \RR$ is interpreted as a parameter.
Generalizing the form of the above \ac{rom} and
$s$ to a parameter $\pp \in \pset \subseteq \CCpar$,
for a positive integer $\npar$,
brings us to the $\Ltwo$-optimal reduced-order modeling problem discussed
in~\cite{MliG23a}.
There, given a parameter-to-output mapping
\begin{align}\label{eq:ymapping}
  \fundef{\yf}{\pset}{\CCoi},
\end{align}
the goal is to construct \iac{strom}
\begin{subequations}\label{eq:rom}
  \begin{align}
    \Ar(\pp) \xr(\pp) & = \Br(\pp),          \\*
    \yr(\pp)          & = \Cr(\pp) \xr(\pp),
  \end{align}
\end{subequations}
with a parameter-separable form
\begin{equation}\label{eq:rom-param-sep-form}
  \Ar(\pp) = \sum_{i = 1}^{\qAr} \car_i(\pp) \cAr_i, \quad
  \Br(\pp) = \sum_{j = 1}^{\qBr} \cbr_j(\pp) \cBr_j, \quad
  \Cr(\pp) = \sum_{k = 1}^{\qCr} \ccr_k(\pp) \cCr_k,
\end{equation}
where
$\xr(\pp) \in \CCri$ is the reduced state,
$\yr(\pp) \in \CCoi$ is the approximate output,
$\Ar(\pp) \in \CCrr$,
$\Br(\pp) \in \CCri$,
$\Cr(\pp) \in \CCor$,
$\fundef{\car_i, \cbr_j, \ccr_k}{\pset}{\CC}$,
$\cAr_i \in \RRrr$,
$\cBr_j \in \RRri$, and
$\cCr_k \in \RRor$.
From here on, we also allow complex reduced matrices, i.e.,
$\cAr_i \in \CCrr$,
$\cBr_j \in \CCri$, and
$\cCr_k \in \CCor$.
We will use the notation $(\cAr_i, \cBr_j, \cCr_k)$ to denote the \ac{strom}
specified by~\eqref{eq:rom} and~\eqref{eq:rom-param-sep-form}.

In~\cite{MliG23a} the \ac{strom}~\eqref{eq:rom} is constructed to minimize
the squared $\Ltwo$ error
\begin{equation}\label{eq:l2-cost}
  \obj\myparen*{\cAr_i, \cBr_j, \cCr_k}
  = \normLtwomu*{\yf - \yr}^2
  = \int_{\pset} \normF*{\yf(\pp) - \yr(\pp)}^2 \difm{\pp},
\end{equation}
where $\measure$ is a measure over $\pset$.

The work~\cite{MliG23a} derived the gradients of $\obj$
with respect to the \emph{real} \ac{strom} matrices $\cAr_i, \cBr_j, \cCr_k$.
To develop the interpolatory necessary optimality conditions of structured
\ac{lti} systems,
we will require diagonalizability of certain matrix pencils
as we will explore in more detail in later sections.
However, since diagonalization of real systems
can lead to complex diagonal matrices,
we need to generalize the results in~\cite{MliG23a,MliG23b}
to complex reduced-order matrices.
The next section achieves this goal using Wirtinger calculus.

\subsection{Wirtinger Calculus and Gradients of the Squared
  \texorpdfstring{$\Ltwo$}{L2} Error}%
\label{sec:l2-grad}
Since the (squared) $\Ltwo$ error is real-valued,
it cannot be analytic with respect to complex variables
(unless it is constant, which is not an interesting case).
Wirtinger calculus generalizes the complex derivative to non-analytic functions;
see~\cite[Section~8.2.1]{Vui14} and references within.
In particular, let the function $f$ of a complex variable $z$ be written as
$f(z) = g(z, \overline{z}) = h(x, y)$
for a function $g$ that is analytic with respect to both variables and
for a function $h$ of the real and imaginary part of $z = x + \imag y$.
Define the following differential operators
\begin{equation*}
  \frac{\partial}{\partial z}
  =
  \frac{1}{2}
  \myparen*{
    \frac{\partial}{\partial x}
    - \imag \frac{\partial}{\partial y}
  }, \quad
  \frac{\partial}{\partial \overline{z}}
  =
  \frac{1}{2}
  \myparen*{
    \frac{\partial}{\partial x}
    + \imag \frac{\partial}{\partial y}
  }.
\end{equation*}
These differential operators allow to differentiate a non-analytic function $f$
by treating $z$ and $\overline{z}$ as independent variables.

This can be analogously generalized to gradients with respect to
complex vectors and matrices.
In particular,
for a real-valued function $f$ of a complex matrix $X$ such that
$f(X) = g(X, \overline{X})$,
if we have that
\begin{equation*}
  f(X + \Delta X)
  = f(X)
  + \ipF*{\nabla_X g(X, \overline{X})}{\overline{\Delta X}}
  + \ipF*{\nabla_{\overline{X}} g(X, \overline{X})}{\Delta X}
  + o(\norm{\Delta X}),
\end{equation*}
then we define $\nabla_X f(X) = \nabla_X g(X, \overline{X})$ and
$\nabla_{\overline{X}} f(X) = \nabla_{\overline{X}} g(X, \overline{X})$.
This allows us to directly generalize the results of~\cite{MliG23a,MliG23b} to
the complex-valued case.
\begin{theorem}\label{thm:grad}
  Suppose that
  $\pset \subseteq \CCpar$,
  $\measure$ is a measure over $\pset$,
  the function $\yf$ is in $\Ltwo(\pset, \measure; \CCoi)$,
  functions $\fundef{\car_i, \cbr_j, \ccr_k}{\pset}{\CC}$ are measurable and
  satisfy
  \begin{equation}\label{eq:abc-l2-bounded}
    \int_{\pset}
    \myparen*{
      \frac{
        \sum_{j = 1}^{\qBr} \abs*{\cbr_j(\pp)}
        \sum_{k = 1}^{\qCr} \abs*{\ccr_k(\pp)}
      }{
        \sum_{i = 1}^{\qAr} \abs*{\car_i(\pp)}
      }
    }^{2}
    \difm{\pp}
    < \infty.
  \end{equation}
  Let $\Sigma_{\textnormal{St}}$ denote the set of \acp{strom}
  $(\cAr_i, \cBr_j, \cCr_k)$ in~\eqref{eq:rom} and~\eqref{eq:rom-param-sep-form}
  such that
  $\cAr_i \in \CCrr$,
  $\cBr_j \in \CCri$,
  $\cCr_k \in \CCor$, and
  \begin{equation}\label{eq:ainv-bounded}
    \esssup_{\pp \in \pset} \, \normF*{\car_i(\pp) {\Ar(\pp)}^{-1}}
    < \infty, \quad
    i = 1, 2, \ldots, \qAr.
  \end{equation}
  Then, for $\obj$ in~\eqref{eq:l2-cost} and any
  $(\cAr_i, \cBr_j, \cCr_k) \in \Sigma_{\textnormal{St}}$,
  \begin{subequations}\label{eq:grad}
    \begin{align}
      \label{eq:grad-A}
      \nabla_{\overline{\cAr_i}} \obj = {}
       &
      \int_{\pset} \overline{\car_i(\pp)}
      \xrd(\pp)
      \myparen*{\yf(\pp) - \yr(\pp)}
      \xr(\pp)\herm
      \difm{\pp},
       & i = 1, 2, \ldots, \qAr, \\
      \label{eq:grad-B}
      \nabla_{\overline{\cBr_j}} \obj = {}
       &
      \int_{\pset} \overline{\cbr_j(\pp)}
      \xrd(\pp)
      \myparen*{\yr(\pp) - \yf(\pp)}
      \difm{\pp},
       & j = 1, 2, \ldots, \qBr, \\
      \label{eq:grad-C}
      \nabla_{\overline{\cCr_k}} \obj = {}
       &
      \int_{\pset} \overline{\ccr_k(\pp)}
      \myparen*{\yr(\pp) - \yf(\pp)}
      \xr(\pp)\herm
      \difm{\pp},
       & k = 1, 2, \ldots, \qCr,
    \end{align}
  \end{subequations}
  where $\xrd(\pp) = \Ar(\pp)\mherm \Cr(\pp)\herm \in \CCro$ is the dual reduced
  state.
\end{theorem}
\begin{proof}
  The proof is similar to the proof of Theorem~3.7 in~\cite{MliG23a}.
  The significant change is that $\ipLtwo{\yf}{\yr}$ is no longer necessarily
  real and not equal to $\ipLtwo{\yr}{\yf}$.
  Therefore, we have
  \begin{equation*}
    \obj
    =
    \normLtwo*{\yf - \yr}^2
    =
    \normLtwo{\yf}^2
    - \ipLtwo*{\yf}{\yr}
    - \ipLtwo*{\yr}{\yf}
    + \normLtwo*{\yr}^2.
  \end{equation*}
  However, since $\ipLtwo{\yf}{\yr}$ does not depend directly on
  $\overline{\cAr_i}, \overline{\cBr_j}, \overline{\cCr_k}$,
  only $\ipLtwo{\yr}{\yf}$ contributes to the gradients.
  Similarly, only one part in the product rule applied to $\ipLtwo{\yr}{\yr}$
  contributes to the gradients.
  This leads to the missing factor of $2$ compared to~\cite{MliG23a}.
  Additionally, because the scalar functions $\car_i, \cbr_j, \ccr_k$
  are not assumed to be closed under conjugation,
  the conjugation remains outside of the function evaluation.
\end{proof}

These gradients can then be used to develop
an $\Ltwo$-optimal reduced-order modeling algorithm,
as done in~\cite{MliG23a}
for the case of real stationary parametric problems and \ac{lti} systems.
In the following, we focus on necessary optimality conditions,
especially in the interpolation form.
In the next subsection, we start with the necessary conditions for \acp{strom}.
This will then be the basis for the structured problems we consider in later
sections.

\subsection{Necessary \texorpdfstring{$\Ltwo$}{L2}-optimality Conditions}
As done in~\cite{MliG23b} for the real-valued case,
an important consequence of \Cref{thm:grad} is that
by setting the gradients to zero,
it yields the necessary optimality conditions
for $\Ltwo$-optimal reduced-order modeling using parameter-separable forms.
The same holds for differential operators from Wirtinger calculus, i.e.,
local minima have zero gradient, leading to our next result.
\begin{corollary}\label{cor:cond}
  Let the assumptions in \Cref{thm:grad} hold.
  Furthermore, let $(\cAr_i, \cBr_j, \cCr_k)$ be an $\Ltwo$-optimal \ac{strom}
  of $H$ among $\Sigma_{\textnormal{St}}$ with $\hH$ as in~\eqref{eq:rom}
  and~\eqref{eq:rom-param-sep-form}.
  Then,
  \begin{subequations}\label{eq:cond}
    \begin{align}
      \label{eq:cond-C}
      \int_{\pset}
      \overline{\ccr_k(\pp)}
      \yf(\pp) \xr(\pp)\herm
      \difm{\pp}
       & =
      \int_{\pset}
      \overline{\ccr_k(\pp)}
      \yr(\pp) \xr(\pp)\herm
      \difm{\pp},
       & k = 1, 2, \ldots, \qCr, \\
      \label{eq:cond-B}
      \int_{\pset}
      \overline{\cbr_j(\pp)}
      \xr_d(\pp) \yf(\pp)
      \difm{\pp}
       & =
      \int_{\pset}
      \overline{\cbr_j(\pp)}
      \xr_d(\pp) \yr(\pp)
      \difm{\pp},
       & j = 1, 2, \ldots, \qBr, \\
      \label{eq:cond-A}
      \int_{\pset}
      \overline{\car_i(\pp)}
      \xr_d(\pp) \yf(\pp) \xr(\pp)\herm
      \difm{\pp}
       & =
      \int_{\pset}
      \overline{\car_i(\pp)}
      \xr_d(\pp) \yr(\pp) \xr(\pp)\herm
      \difm{\pp},
       & i = 1, 2, \ldots, \qAr.
    \end{align}
  \end{subequations}
\end{corollary}
In the next subsection,
we specialize these conditions to \acp{diagstrom}.

\subsection{\texorpdfstring{$\Ltwo$}{L2}-optimality for Diagonal Structured
  Dynamics}%
\label{sec:l2-diag}
We are interested in deriving interpolatory optimality conditions and
this will require a certain diagonalizability assumption
(as also used in the unstructured $\Htwo$-optimality
conditions~\eqref{eq:lti-oc}).
In particular, we require that there exist matrices $\hS, \hT \in \CCrr$
such that $\hS\herm \cAr_i \hT$ is diagonal for all $i = 1, 2, \ldots, \qAr$.

This is generically true when $\qAr \le 2$,
as in the case of non-parametric \ac{lti} systems~\eqref{eq:lti-rom}.
Such systems were considered in~\cite{MliG23b},
as well as stationary parametric systems with $\qAr = 2$ and
parametric \ac{lti} systems with a special tensor structure.
Here, we consider systems with $\qAr \ge 3$,
such as second-order and time-delay systems,
where diagonalizability is no longer a generic property.

Therefore, we consider \acp{diagstrom}, i.e.,
\acp{strom} $(\cAr_i, \cBr_j, \cCr_k)$ defined by~\eqref{eq:rom}
and~\eqref{eq:rom-param-sep-form}
but with diagonal $\cAr_i$ (thus we are assuming that the transformation
$\hS\herm \cAr_i \hT$ has been already applied).
As mentioned, a generic example of such systems are those with $\qAr = 2$ after
a diagonalizing transformation.
An additional example are second-order systems with Rayleigh damping where
$\Ar(s) = s^2 \hM + s (\alpha \hM + \beta \hK) + \hK$ and
$\hM$ and $\hK$ are the mass and stiffness matrices, respectively.

Deriving the optimality conditions for \acp{diagstrom} will require computing
gradients with respect to diagonal matrices.
The following lemma will allow us to reuse the gradients
with respect to full matrices
to find gradients with respect to diagonal matrices.
\begin{lemma}\label{lem:diag}
  Let $V$ be $\CC^{m}$ or $\CC^{m \times n}$, and
  $\fundef{f}{V}{\RR}$ be a Wirtinger differentiable function.
  Furthermore, let $W \subseteq V$ be a subspace closed under conjugation and
  define $g = f\vert_W$ as the restriction of $f$ onto $W$.
  Then for $w \in W$,
  the function $g$ is Wirtinger differentiable at $w$ and
  $\nabla_{\overline{w}} g(w) = \Proj_W \nabla_{\overline{w}} f(w)$
  where $\fundef{\Proj_W}{V}{W}$ is the orthogonal projector onto $W$.

  In particular,
  let $\fundef{F}{\CCrr}{\RR}$ be a Wirtinger differentiable function and
  $G = F\vert_D$ its restriction to diagonal matrices
  $D = \{X \in \CCrr : X_{ij} = 0,\ \forall i \neq j\}$.
  Then for $X \in D$, the function $G$ is Wirtinger differentiable at $X$ and
  $\nabla_{\overline{X}} G(X) = \mydiag{\nabla_{\overline{X}} F(X)}$,
  where $\mydiag{Y} = \sum_{i = 1}^{\nrom} Y_{ii} e_i e_i\tran$
  is the diagonal part of the matrix $Y$.
\end{lemma}
\begin{proof}
  Since $f$ is Wirtinger differentiable at $w$,
  we have that,
  for $\Delta w \in W$,
  \begin{equation*}
    f(w + \Delta w)
    =
    f(w)
    + \ip*{\nabla_{w} f(w)}{\overline{\Delta w}}
    + \ip*{\nabla_{\overline{w}} f(w)}{\Delta w}
    + o(\norm{\Delta w}).
  \end{equation*}
  Therefore, since $f$ and $g$ are equal over $W$, we have
  \begin{equation}
    \label{eq:g-diff}
    g(w + \Delta w)
    =
    g(w)
    + \ip*{\nabla_{w} f(w)}{\overline{\Delta w}}
    + \ip*{\nabla_{\overline{w}} f(w)}{\Delta w}
    + o(\norm{\Delta w}).
  \end{equation}
  We observe that
  \begin{align*}
    \ip*{\nabla_{w} f(w)}{\overline{\Delta w}}
     & =
    \ip*{\Proj_W \nabla_{w} f(w)}{\overline{\Delta w}}
    + \ip*{\nabla_{w} f(w) - \Proj_W \nabla_{w} f(w)}{\overline{\Delta w}} \\
     & =
    \ip*{\Proj_W \nabla_{w} f(w)}{\overline{\Delta w}}
  \end{align*}
  and
  \begin{align*}
    \ip*{\nabla_{\overline{w}} f(w)}{\Delta w}
     & =
    \ip*{\Proj_W \nabla_{\overline{w}} f(w)}{\Delta w}
    + \ip*{\nabla_{\overline{w}} f(w)
    - \Proj_W \nabla_{\overline{w}} f(w)}{\Delta w} \\
     & =
    \ip*{\Proj_W \nabla_{\overline{w}} f(w)}{\Delta w}.
  \end{align*}
  Thus,~\eqref{eq:g-diff} becomes
  \begin{equation*}
    g(w + \Delta w)
    =
    g(w)
    + \ip*{\Proj_W \nabla_{w} f(w)}{\overline{\Delta w}}
    + \ip*{\Proj_W \nabla_{\overline{w}} f(w)}{\Delta w}
    + o(\norm{\Delta w}).
  \end{equation*}
  Since $\Proj_W \nabla_{w} f(w)$ and $\Proj_W \nabla_{\overline{w}} f(w)$ are
  elements of $W$, it follows that
  $g$ is Wirtinger differentiable at $w$, and its gradients are
  $\nabla_{w} g(w) = \Proj_W \nabla_{w} f(w)$ and
  $\nabla_{\overline{w}} g(w) = \Proj_W \nabla_{\overline{w}} f(w)$,
  which proves the first part.

  Next, we notice that the set of diagonal matrices $D$
  is a subspace of the space of matrices $\CCrr$,
  it is closed under conjugation, and
  $\Proj_D(Y) = \mydiag{Y}$.
  Thus, the second part of the result directly follows from the first.
\end{proof}

\Cref{lem:diag} will now help us derive the necessary $\Ltwo$-optimality
conditions for \acp{diagstrom}.
Since in this case all $\cAr_i$'s are diagonal
(and thus in return so is $\Ar(\pp)$ in~\eqref{eq:rom}),
$\yr$ for a \ac{diagstrom} has a ``pole-residue'' form, i.e.,
\begin{equation}\label{eq:diag-pole-res}
  \yr(\pp)
  = \Cr(\pp) {\Ar(\pp)}^{-1} \Br(\pp)
  = \sum_{\ell = 1}^{\nrom}
  \frac{c_{\ell}(\pp) b_{\ell}(\pp)\herm}{a_{\ell}(\pp)},
\end{equation}
where
$a_{\ell}(\pp)$ is the $\ell$th diagonal entry of $\Ar(\pp)$,
$b_{\ell}(\pp) \in \CCi$ is the $\ell$th column of $\Br(\pp)\herm$, and
$c_{\ell}(\pp) \in \CCo$ is the $\ell$th column of $\Cr(\pp)$.
With this pole-residue form in hand,
we now can investigate the optimality conditions for \acp{diagstrom}.
\begin{corollary}\label{cor:cond-diag}
  Let the assumptions in \Cref{thm:grad} hold and
  denote by $\Sigma_{\textnormal{D-St}}$ the set of \acp{diagstrom} in
  $\Sigma_{\textnormal{St}}$.
  Furthermore, let $(\cAr_i, \cBr_j, \cCr_k)$ be an $\Ltwo$-optimal
  \ac{diagstrom} of $H$ among $\Sigma_{\textnormal{D-St}}$ with $\hH$ as
  in~\eqref{eq:diag-pole-res}.
  Then
  \begin{subequations}\label{eq:cond-diag}
    \begin{align}
      \label{eq:cond-C-diag}
      \int_{\pset}
      \frac{
        \overline{\ccr_k(\pp)}
        \yf(\pp)
        b_{\ell}(\pp)
      }{
        \overline{a_{\ell}(\pp)}
      }
      \difm{\pp}
       & =
      \int_{\pset}
      \frac{
        \overline{\ccr_k(\pp)}
        \yr(\pp)
        b_{\ell}(\pp)
      }{
        \overline{a_{\ell}(\pp)}
      }
      \difm{\pp},
       & k = 1, 2, \ldots, \qCr, \\
      \label{eq:cond-B-diag}
      \int_{\pset}
      \frac{
        \overline{\cbr_j(\pp)}
        c_{\ell}(\pp)\herm
        \yf(\pp)
      }{
        \overline{a_{\ell}(\pp)}
      }
      \difm{\pp}
       & =
      \int_{\pset}
      \frac{
        \overline{\cbr_j(\pp)}
        c_{\ell}(\pp)\herm
        \yr(\pp)
      }{
        \overline{a_{\ell}(\pp)}
      }
      \difm{\pp},
       & j = 1, 2, \ldots, \qBr, \\
      \label{eq:cond-A-diag}
      \int_{\pset}
      \frac{
        \overline{\car_i(\pp)}
        c_{\ell}(\pp)\herm
        \yf(\pp)
        b_{\ell}(\pp)
      }{
        \overline{a_{\ell}(\pp)}^2
      }
      \difm{\pp}
       & =
      \int_{\pset}
      \frac{
        \overline{\car_i(\pp)}
        c_{\ell}(\pp)\herm
        \yr(\pp)
        b_{\ell}(\pp)
      }{
        \overline{a_{\ell}(\pp)}^2
      }
      \difm{\pp},
       & i = 1, 2, \ldots, \qAr,
    \end{align}
  \end{subequations}
  for $\ell = 1, 2, \ldots, \nrom$ where the coefficients
  $\car_i$, $\cbr_j$, and $\ccr_k$ are as defined
  in~\eqref{eq:rom-param-sep-form}.
\end{corollary}
\begin{proof}
  The right and left tangential conditions~\eqref{eq:cond-C-diag}
  and~\eqref{eq:cond-B-diag} follow, respectively,
  from conditions~\eqref{eq:cond-C} and~\eqref{eq:cond-B} of \Cref{cor:cond},
  using the facts that
  \begin{equation*}
    \xr(\pp)\herm
    =
    \Br(\pp)\herm
    \Ar(\pp)\mherm
    =
    \begin{bmatrix}
      \frac{b_1(\pp)}{a_1(\pp)\herm}
       & \cdots
       & \frac{b_{\nrom}(\pp)}{a_{\nrom}(\pp)\herm}
    \end{bmatrix}\!,\quad
    \xrd(\pp)
    =
    \Ar(\pp)\mherm
    \Cr(\pp)\herm
    =
    \begin{bmatrix}
      \frac{c_1(\pp)\herm}{a_1(\pp)\herm} \\
      \vdots                              \\
      \frac{c_{\nrom}(\pp)\herm}{a_{\nrom}(\pp)\herm}
    \end{bmatrix}\!,
  \end{equation*}
  since $\Ar(\pp)$ is diagonal.
  For the bitangential conditions~\eqref{eq:cond-A-diag},
  we cannot directly use the condition~\eqref{eq:cond-A},
  but we can use the gradient~\eqref{eq:grad-A}
  together with \Cref{lem:diag},
  to conclude that the gradient with respect to the diagonal
  is the diagonal part of the gradient.
  Setting the diagonal entries to zeros gives us~\eqref{eq:cond-A-diag}.
\end{proof}
\Cref{thm:grad} and \Cref{cor:cond-diag} will form the foundation of our
analysis and allow us to derive the interpolatory optimality conditions for
structured \ac{lti} systems by choosing the variables
(such as $a_\ell(\pp)$, $\car_i(\pp)$ etc.)
suitably based on the various \ac{lti} system structures we consider.
In particular,
these results will be used for the non-parametric structured $\Htwo$ cases in
\Cref{sec:h2-diag,sec:so,sec:ph,sec:td}.
\begin{remark}\label{rem:opt-orth}
  The necessary optimality conditions in \Cref{cor:cond-diag} can be interpreted
  as orthogonality conditions:
  \begin{align*}
    \ipLtwo*{
    \frac{
    \ccr_k(\pp)
    e_m^{\nout}
    b_{\ell}(\pp)\herm
    }{a_{\ell}(\pp)}
    }{\yf - \yr}
     & = 0,
     & k = 1, 2, \ldots, \qCr,\
    m = 1, 2, \ldots, \nout,    \\
    \ipLtwo*{
    \frac{
    \cbr_j(\pp)
    c_{\ell}(\pp)
    {(e_m^{\nin})}\herm
    }{a_{\ell}(\pp)}
    }{\yf - \yr}
     & = 0,
     & j = 1, 2, \ldots, \qBr,\
    m = 1, 2, \ldots, \nin,     \\
    \ipLtwo*{
      \frac{
        \car_i(\pp)
        c_{\ell}(\pp)
        b_{\ell}(\pp)\herm
      }{{a_{\ell}(\pp)}^2}
    }{\yf - \yr}
     & = 0,
     & i = 1, 2, \ldots, \qAr,
  \end{align*}
  for $\ell = 1, 2, \ldots, \nrom$,
  where $e_m^{n} \in \RR^n$ is the $m$th canonical basis vector.
  This directly generalizes the same property for unstructured \ac{lti} systems.
\end{remark}

\section{Interpolatory \texorpdfstring{$\Htwo$}{H2}-optimality Conditions for
  Diagonal Structured Dynamics}%
\label{sec:h2-diag}
We now turn our attention back to
non-parametric (yet structured)
\ac{lti} systems where, in the language of \Cref{sec:l2},
the parameter space is $\pset = \imag \RR$,
the parameter is $\pp = s = \imag \omega$, and
the measure $\measure$ is the Lebesgue measure over $\imag \RR$.
For simplicity, we only consider \ac{lti} systems with
$s$-independent $\Br$ and $\Cr$.
Additionally, as in \Cref{sec:l2-diag},
we are interested in \ac{lti} systems with diagonal structure.

Recall that
the well-established interpolatory optimality conditions~\eqref{eq:lti-oc}
for unstructured \ac{lti} systems~\eqref{eq:lti-rom}
assume diagonalizability of $\hA$,
which leads to the pole-residue form~\eqref{eq:pole-res-unstruct}
with the denominator of the simple form $s - \lambda_i$.
Here, we derive the interpolatory $\Htwo$-optimality conditions for structured
transfer functions with a much more general (structured) denominator,
which will be needed for particular types of structured \ac{lti} systems
in later sections.

The pole-residue decomposition of a rational function
plays the fundamental role in deriving the $\Htwo$-optimality conditions
for the unstructured problem.
For structured problems, the residue computations are more involved.
The next result will be useful in residue computation and
then in simplifying the integrals in \Cref{cor:cond-diag}
for the structured problems under consideration.
\begin{lemma}\label{lem:res}
  Let $g$ and $h$ be analytic
  in a neighborhood of $c \in \CC$
  such that $h(c) = 0$ and $h'(c) \neq 0$.
  Define $f_1$ and $f_2$ as
  $f_1(z) = \frac{g(z)}{h(z)}$ and
  $f_2(z) = \frac{g(z)}{{h(z)}^2}$
  in a neighborhood of $c$.
  Then
  \begin{equation}\label{eq:resf1f2}
    \Res(f_1, c)
    =
    \frac{g(c)}{h'(c)}
    \quad \text{and} \quad
    \Res(f_2, c)
    =
    \frac{g'(c)}{{h'(c)}^2}
    - \frac{g(c) h''(c)}{{h'(c)}^3}.
  \end{equation}
\end{lemma}
\begin{proof}
  The first equality follows directly from~\cite[Proposition~4.1.2]{MarH99}.
  For the second equality, using~\cite[Proposition~4.1.4]{MarH99},
  it follows that
  \begin{align*}
    \Res(f_2, c)
     & =
    2 \frac{g'(c)}{h_2''(c)}
    - \frac{2 g(c) h_2'''(c)}{3 {h_2''(c)}^2},
  \end{align*}
  where $h_2(z) = {h(z)}^2$.
  From
  $h_2'(z) = 2 h(z) h'(z)$,
  $h_2''(z) = 2 {h'(z)}^2 + 2 h(z) h''(z)$, and
  $h_2'''(z) = 6 h'(z) h''(z) + 2 h(z) h'''(z)$,
  we obtain
  \begin{align*}
    \Res(f_2, c)
     & =
    2 \frac{g'(c)}{2 {h'(c)}^2}
    - \frac{12 g(c) h'(c) h''(c)}{3 {[2 {h'(c)}^2]}^2}
    =
    \frac{g'(c)}{{h'(c)}^2}
    - \frac{g(c) h''(c)}{{h'(c)}^3},
  \end{align*}
  which completes the proof.
\end{proof}
Furthermore, we need a family of contours to apply the Residue Theorem with
potentially infinitely many poles so that our theory applies to a much more
general class of systems than the finite-dimensional ones.
\begin{definition}\label{def:contours}
  Let $\fundef{\car_i}{\CC}{\CC}$ be entire and
  $\fundef{a}{\CC}{\CC}$ a linear combination of $\car_i$.
  We call a sequence of closed contours $C_j$, $j = 1, 2, \ldots$,
  \emph{admissible} if
  \begin{enumerate*}[label=(\alph*)]
    \item $C_1$ contains the origin in its interior,
    \item $C_j$ is contained in the interior of $C_{j + 1}$ for all $j \ge 1$,
    \item $C_j$ does not pass through any of the zeros of $a$ for all $j \ge 1$,
    \item $C_j$ crosses the imaginary axis exactly twice for all $j \ge 1$,
    \item $\dist{0, C_j} \to \infty$ as $j \to \infty$,
    \item $\len{C_j} \max_{s \in C_j} \frac{1}{\abs{a(s)}}$ and
    $\len{C_j} \max_{s \in C_j} \frac{\abs{\car_i(s)}}{\abs{a(s)}^2}$
    are bounded, and
    \item for all sufficiently large $j$,
    there is at most one zero of $a$ between $C_j$ and
    $C_{j + 1}$.
  \end{enumerate*}
\end{definition}
Note that, if $\car_i$ are polynomials (i.e., have finitely many poles),
one of which is of degree at least one, and $\deg a \ge \deg \car_i$,
then a sequence of circles with
a common center at the origin,
increasing radius, and
containing all zeros of $a$
will immediately satisfy the assumptions of the above definition.
Additionally, for time-delay systems we consider in \Cref{sec:td},
there exist admissible sequences of contours.

Applying \Cref{lem:res} to \Cref{cor:cond-diag} yields the following result.
\begin{theorem}\label{thm:cond-diag-s}
  Let $H \in \Htwo^{\nout \times \nin}$.
  Furthermore, let $\fundef{\car_i}{\CC}{\CC}$ be entire and
  \begin{equation}\label{eq:a-l2-bounded}
    \int_{-\infty}^{\infty}
    \frac{1}{
      \myparen*{
        \sum_{i = 1}^{\qAr} \abs*{\car_i(\imag \omega)}
      }^{2}
    }
    \dif{\omega}
    < \infty.
  \end{equation}
  Define $\Sigma_{\textnormal{D-St-LTI}}$ as the set of \acp{diagstrom}
  $(\cAr_i, \cBr, \cCr)$ with parameter-independent $\Br$ and $\Cr$ such that
  $\cAr_i \in \CCrr$,
  $\cBr \in \CCri$,
  $\cCr \in \CCor$,
  $a_{\ell}$ has zeros only in $\CC_-$, and
  \begin{equation}\label{eq:a-linf-bounded}
    \sup_{\omega \in \RR} \,
    \frac{
      \abs*{\car_i(\imag \omega)}
    }{
      \abs*{a_{\ell}(\imag \omega)}
    }
    < \infty, \quad
    i = 1, 2, \ldots, \qAr,\
    \ell = 1, 2, \ldots, \nrom,
  \end{equation}
  where $a_{\ell}(s)$ is the $\ell$th diagonal entry of $\Ar(s)$ as defined
  in~\eqref{eq:rom-param-sep-form}.
  Let $(\cAr_i, \cBr, \cCr)$ be an $\Htwo$-optimal \ac{diagstrom} of $H$ in
  $\Sigma_{\textnormal{D-St-LTI}}$.
  Let there exist admissible contours $C_j^{(\ell)}$ for $a_{\ell}$,
  $\ell = 1, 2, \ldots, \nrom$ (see \Cref{def:contours}).
  Additionally, let $\Lambda_{\ell}$ be the set of zeros of $a_{\ell}$,
  ordered according to the contours $C_j^{(\ell)}$.
  Assume that all the zeros of all $a_{\ell}$ are simple and pairwise distinct.
  Let $b_{\ell} = \cBr\herm e_{\ell}$ and $c_{\ell} = \cCr e_{\ell}$.
  Then
  \begin{equation}\label{eq:pole-res-h2-diag}
    \hH(s) = \Cr {\Ar(s)}^{-1} \Br
    = \sum_{\ell = 1}^{\nrom}
    \frac{c_{\ell} b_{\ell}\herm}{a_{\ell}(s)}
  \end{equation}
  and it satisfies the interpolatory optimality conditions
  \begin{subequations}
    \begin{gather}
      \label{eq:cond-diag-s-1}
      \sum_{\lambda \in \Lambda_{\ell}}
      \frac{
        H\myparen*{-\overline{\lambda}}
        b_{\ell}
      }{
        \overline{a_{\ell}'(\lambda)}
      }
      =
      \sum_{\lambda \in \Lambda_{\ell}}
      \frac{
        \hH\myparen*{-\overline{\lambda}}
        b_{\ell}
      }{
        \overline{a_{\ell}'(\lambda)}
      }, \\
      \label{eq:cond-diag-s-2}
      \sum_{\lambda \in \Lambda_{\ell}}
      \frac{
        c_{\ell}\herm
        H\myparen*{-\overline{\lambda}}
      }{
        \overline{a_{\ell}'(\lambda)}
      }
      =
      \sum_{\lambda \in \Lambda_{\ell}}
      \frac{
        c_{\ell}\herm
        \hH\myparen*{-\overline{\lambda}}
      }{
        \overline{a_{\ell}'(\lambda)}
      }, \\
      \label{eq:cond-diag-s-3}
      \begin{aligned}
         &
        \sum_{\lambda \in \Lambda_{\ell}}
        c_{\ell}\herm
        \myparen*{
          \overline{\myparen*{\frac{\car_i(\lambda)}{{a_{\ell}'(\lambda)}^2}}}
          H'\myparen*{-\overline{\lambda}}
          -
          \overline{
            \myparen*{
              \frac{\car_i'(\lambda)}{{a_{\ell}'(\lambda)}^2}
              -
              \frac{\car_i(\lambda) a_{\ell}''(\lambda)}{
                {a_{\ell}'(\lambda)}^3}
            }
          }
          H\myparen*{-\overline{\lambda}}
        }
        b_{\ell} \\*
         & =
        \sum_{\lambda \in \Lambda_{\ell}}
        c_{\ell}\herm
        \myparen*{
          \overline{\myparen*{\frac{\car_i(\lambda)}{{a_{\ell}'(\lambda)}^2}}}
          \hH'\myparen*{-\overline{\lambda}}
          -
          \overline{
            \myparen*{
              \frac{\car_i'(\lambda)}{{a_{\ell}'(\lambda)}^2}
              -
              \frac{\car_i(\lambda) a_{\ell}''(\lambda)}{
                {a_{\ell}'(\lambda)}^3}
            }
          }
          \hH\myparen*{-\overline{\lambda}}
        }
        b_{\ell}.
      \end{aligned}
    \end{gather}
  \end{subequations}
  for $\ell = 1, 2, \ldots, \nrom$ and $i = 1, 2, \ldots, \qAr$.
\end{theorem}
\begin{proof}
  The ``pole-residue'' form~\eqref{eq:pole-res-h2-diag} directly follows
  from~\eqref{eq:diag-pole-res}.
  The condition~\eqref{eq:a-l2-bounded} is equivalent
  to~\eqref{eq:abc-l2-bounded} since $\Br$ and $\Cr$ are parameter-independent.
  Similarly, the condition~\eqref{eq:a-linf-bounded} is equivalent
  to~\eqref{eq:ainv-bounded} since $\Ar(s)$ is diagonal.
  Thus, we can apply \Cref{cor:cond-diag}.
  Then, the condition~\eqref{eq:cond-C-diag} becomes
  \begin{equation}\label{eq:cond-diag-proof-1}
    \int_{-\infty}^{\infty}
    \frac{H(\imag \omega) b_{\ell}}{\overline{a_{\ell}}(-\imag \omega)}
    \dif{\omega}
    =
    \int_{-\infty}^{\infty}
    \frac{\hH(\imag \omega) b_{\ell}}{\overline{a_{\ell}}(-\imag \omega)}
    \dif{\omega},
  \end{equation}
  where $\overline{a_{\ell}}(s) = \overline{a_{\ell}(\bar{s})}$ is analytic.
  Let $\gamma_j$ be a negatively-oriented closed curve consisting of two curves:
  an interval along the imaginary axis $[-\imag r_{j, 1}, \imag r_{j, 2}]$
  denoted by $\gamma_{j, 1}$ and the part of $-\overline{C_j^{(\ell)}}$ in the
  right half-plane denoted by $\gamma_{j, 2}$.
  Then
  \begin{equation*}
    \int_{-r_{j, 1}}^{r_{j, 2}}
    \frac{H(\imag \omega) b_{\ell}}{\overline{a_{\ell}}(-\imag \omega)}
    \dif{\omega}
    =
    \frac{1}{\imag}
    \oint_{\gamma_j}
    \frac{H(s) b_{\ell}}{\overline{a_{\ell}}(-s)}
    \dif{s}
    -
    \frac{1}{\imag}
    \oint_{\gamma_{j, 2}}
    \frac{H(s) b_{\ell}}{\overline{a_{\ell}}(-s)}
    \dif{s}.
  \end{equation*}
  Using the Residue Theorem and \Cref{lem:res}, we obtain
  \begin{equation}\label{eq:cond-diag-proof-2}
    \int_{-r_{j, 1}}^{r_{j, 2}}
    \frac{H(\imag \omega) b_{\ell}}{\overline{a_{\ell}}(-\imag \omega)}
    \dif{\omega}
    =
    -2 \pi
    \sum_{\substack{\lambda \in \Lambda_{\ell} \\
    \lambda \text{ inside } C_j^{\ell}}}
    \frac{
      H\myparen*{-\overline{\lambda}}
      b_{\ell}
    }{
      \overline{a_{\ell}}'\myparen*{\overline{\lambda}}
    }
    -
    \frac{1}{\imag}
    \oint_{\gamma_{j, 2}}
    \frac{H(s) b_{\ell}}{\overline{a_{\ell}}(-s)}
    \dif{s}.
  \end{equation}
  Next, we observe that
  \begin{align*}
    \norm*{
      \oint_{\gamma_{j, 2}}
      \frac{H(s) b_{\ell}}{\overline{a_{\ell}}(-s)}
      \dif{s}
    }
     & \le
    \len{\gamma_{j, 2}}
    \max_{s \in \gamma_{j, 2}}
    \norm*{
      \frac{H(s) b_{\ell}}{\overline{a_{\ell}}(-s)}
    }      \\
     & \le
    \len{C_j^{(\ell)}}
    \max_{s \in C_j^{(\ell)}}
    \frac{1}{\abs{a_{\ell}(s)}}
    \max_{s \in \gamma_{j, 2}}
    \norm*{H(s)}
    \norm*{b_{\ell}}.
  \end{align*}
  Since $H$ is an $\Htwo$ function, we have that
  \(
  \max_{s \in \gamma_{j, 2}}
  \norm*{H(s)}
  \)
  converges to zero as $j \to \infty$.
  Therefore, applying $\lim_{j \to \infty}$ to both sides
  in~\eqref{eq:cond-diag-proof-2} gives us
  \begin{equation*}
    \int_{-\infty}^{\infty}
    \frac{H(\imag \omega) b_{\ell}}{\overline{a_{\ell}}(-\imag \omega)}
    \dif{\omega}
    =
    -2 \pi
    \sum_{\lambda \in \Lambda_{\ell}}
    \frac{
      H\myparen*{-\overline{\lambda}}
      b_{\ell}
    }{
      \overline{a_{\ell}}'\myparen*{\overline{\lambda}}
    }.
  \end{equation*}
  Deriving the same expression for $\hH$ and inserting both
  into~\eqref{eq:cond-diag-proof-1} implies condition~\eqref{eq:cond-diag-s-1}.
  Condition~\eqref{eq:cond-diag-s-2} is obtained in a similar way.

  Condition~\eqref{eq:cond-A-diag} becomes
  \begin{align*}
    \int_{-\infty}^{\infty}
    \frac{
      \overline{\car_i}(-\imag \omega)
      c_{\ell}\herm
      H(\imag \omega)
      b_{\ell}
    }{
      {\overline{a_{\ell}}(-\imag \omega)}^2
    }
    \dif{\omega}
     & =
    \int_{-\infty}^{\infty}
    \frac{
      \overline{\car_i}(-\imag \omega)
      c_{\ell}\herm
      \hH(\imag \omega)
      b_{\ell}
    }{
      {\overline{a_{\ell}}(-\imag \omega)}^2
    }
    \dif{\omega},
     & i = 1, 2, \ldots, \qAr,
  \end{align*}
  where $\overline{\car_i}(s) = \overline{\car_i(\bar{s})}$ is analytic.
  Following the same procedure as above and noting that
  \begin{equation*}
    \abs*{
      \oint_{\gamma_{j, 2}}
      \frac{
        \overline{\car_i}(-s)
        c_{\ell}\herm
        H(s)
        b_{\ell}
      }{
        {\overline{a_{\ell}}(-s)}^2
      }
      \dif{s}
    }
    \le
    \len{C_j^{(\ell)}}
    \max_{s \in C_j^{(\ell)}}
    \frac{\abs{\car_i(s)}}{\abs{a_{\ell}(s)}^2}
    \max_{s \in \gamma_{j, 2}}
    \norm*{H(s)}
    \norm*{b_{\ell}}
    \norm*{c_{\ell}}
  \end{equation*}
  converges to zero as $j \to \infty$, we obtain that
  \begin{align*}
     &
    \int_{-\infty}^{\infty}
    \frac{
      \overline{\car_i}(-\imag \omega)
      c_{\ell}\herm
      H(\imag \omega)
      b_{\ell}
    }{
      {\overline{a_{\ell}}(-\imag \omega)}^2
    }
    \dif{\omega} \\
     & =
    -2 \pi
    \sum_{\lambda \in \Lambda_{\ell}}
    c_{\ell}\herm
    \myparen*{
      \frac{
        -\overline{\car_i}'\myparen*{\overline{\lambda}}
        H\myparen*{-\overline{\lambda}}
        + \overline{\car_i}\myparen*{\overline{\lambda}}
        H'\myparen*{-\overline{\lambda}}
      }{
        \overline{a_{\ell}}'\myparen*{\overline{\lambda}}^2
      }
      +
      \frac{
        \overline{\car_i}\myparen*{\overline{\lambda}}
        \overline{a_{\ell}}''\myparen*{\overline{\lambda}}
        H\myparen*{-\overline{\lambda}}
      }{
        \overline{a_{\ell}}'\myparen*{\overline{\lambda}}^3
      }
    }
    b_{\ell}
  \end{align*}
  After simplifying and using the same expression for $\hH$, we have
  \begin{align*}
     &
    \sum_{\lambda \in \Lambda_{\ell}}
    c_{\ell}\herm
    \myparen*{
      \frac{
        \overline{\car_i}\myparen*{\overline{\lambda}}
        H'\myparen*{-\overline{\lambda}}
      }{
        \overline{a_{\ell}}'\myparen*{\overline{\lambda}}^2
      }
      -
      \myparen*{
        \frac{
          \overline{\car_i}'\myparen*{\overline{\lambda}}
        }{
          \overline{a_{\ell}}'\myparen*{\overline{\lambda}}^2
        }
        -
        \frac{
          \overline{\car_i}\myparen*{\overline{\lambda}}
          \overline{a_{\ell}}''\myparen*{\overline{\lambda}}
        }{
          \overline{a_{\ell}}'\myparen*{\overline{\lambda}}^3
        }
      }
      H\myparen*{-\overline{\lambda}}
    }
    b_{\ell} \\*
     & =
    \sum_{\lambda \in \Lambda_{\ell}}
    c_{\ell}\herm
    \myparen*{
      \frac{
        \overline{\car_i}\myparen*{\overline{\lambda}}
        \hH'\myparen*{-\overline{\lambda}}
      }{
        \overline{a_{\ell}}'\myparen*{\overline{\lambda}}^2
      }
      -
      \myparen*{
        \frac{
          \overline{\car_i}'\myparen*{\overline{\lambda}}
        }{
          \overline{a_{\ell}}'\myparen*{\overline{\lambda}}^2
        }
        -
        \frac{
          \overline{\car_i}\myparen*{\overline{\lambda}}
          \overline{a_{\ell}}''\myparen*{\overline{\lambda}}
        }{
          \overline{a_{\ell}}'\myparen*{\overline{\lambda}}^3
        }
      }
      \hH\myparen*{-\overline{\lambda}}
    }
    b_{\ell}.
  \end{align*}
  This gives condition~\eqref{eq:cond-diag-s-3}.
\end{proof}
In the classical unstructured case~\eqref{eq:lti-rom},
we have $a_{\ell}(s) = s - \lambda_{\ell}$.
Since $a_{\ell}'(s) = 1$ and $a_{\ell}''(s) = 0$, and
we see that \Cref{thm:cond-diag-s} recovers the known result~\eqref{eq:lti-oc}
as a special case.
Thus, we have extended the well-known interpolatory $\Htwo$-optimality
conditions for classical \ac{lti} systems with the pole-residue
form~\eqref{eq:pole-res-unstruct} to
structured \ac{lti} systems with a generalized
``pole-residue''-like form~\eqref{eq:pole-res-h2-diag}, allowing structured
forms to be encoded in the denominator $a_\ell(s)$.
In the following sections,
by adjusting $a_\ell(s)$ to the structures of interest,
we will develop new interpolatory $\Htwo$ optimality conditions for important
classes of structured \ac{lti} systems.

\section{Second-order Systems}%
\label{sec:so}
In this section, we consider a prominent class of \acp{strom},
namely the \ac{lti} second-order systems of the form
\begin{subequations}\label{eq:so-sys-r}
  \begin{align}
    \ddot{\hx}(t) + \hE \dot{\hx}(t) + \hK \hx(t) & = \hB u(t),   \\*
    \hy(t)                                        & = \hC \hx(t),
  \end{align}
\end{subequations}
where
$\hE, \hK \in \RRff$ are, respectively, the damping and stiffness matrices;
$\hB \in \RRfi$ is the input-to-state map;
$\hC \in \RRof$ is the state-to-output map;
$\hx(t) \in \RRf$ is the internal state;
$u(t) \in \RRi$ are the inputs, and
$\hy(t) \in \RRo$ are outputs.
It is common to also include an invertible mass matrix $\hM$ multiplying the
second derivative $\ddot{\hx}$,
but here we take it to be the identity matrix without loss of generality.
We denote such systems by $(\hE, \hK, \hB, \hC)$.
These systems appear frequently, e.g.,
in analyzing mechanical or electrical systems
(see, e.g.,~\cite{ReiS08} and references within).
The transfer function of~\eqref{eq:so-sys-r} is given by
\begin{equation*}
  \hH(s) = \hC \myparen*{s^2 \hI + s \hE + \hK}^{-1} \hB.
\end{equation*}
We assume that the matrix pencil $\lambda^2 \hI + \lambda \hE + \hK$
is asymptotically stable,
which makes the properties in~\eqref{eq:a-l2-bounded}
and~\eqref{eq:a-linf-bounded} true.
Again, the transfer function $H$ of the \ac{fom} can be associated to a
finite-dimensional, second-order system,
but that is not necessary in the following
(it is enough that it has a finite $\Htwo$ norm).

There are various approaches to model reduction of second-order systems
using different error measures.
We refer the reader to~\cite{Wer21,SaaSW19} for an overview.
Here our focus is on interpolatory methods,
more precisely in establishing the interpolatory conditions
for minimizing the $\Htwo$ error~\eqref{eq:h2-error}.

Structure-preserving interpolatory reduced-order modeling of second-order
systems has been studied in detail,
see, e.g.,~\cite{BonFSZ16,BeaG09,ChaGVV05,BaiS05,Bai02,BeaG05} and
the references therein.
Inspired by \ac{irka}~\cite{GugAB08} for $\Htwo$-optimal reduced-order modeling
of unstructured systems,
Wyatt~\cite{Wya12} proposed several iterative methods
for reducing second-order dynamics with a focus on the $\Htwo$ norm.
However, the second-order $\Htwo$-optimality conditions were not established and
thus the resulting reduced models did not satisfy
the true optimality conditions.
The interpolatory $\Htwo$-optimality conditions
for reducing second-order systems
were derived by Beattie and Benner~\cite{BeaB14}
where the \ac{rom} was assumed to be modally damped, i.e.,
$\hE$ and $\hK$ are symmetric positive definite and
simultaneously diagonalizable.

\subsection{Interpolatory Conditions}

Starting with the general optimality conditions in \Cref{thm:cond-diag-s} and
by appropriately choosing the parameters to reflect the second-order dynamics,
we first derive the interpolatory optimality conditions for $\Htwo$-optimal
reduced-order modeling of second-order systems for the special case of
modally-damped systems.
In other words,
we show that the optimality conditions of~\cite[Section~5]{BeaB14}
follow as a special case of our general framework from \Cref{thm:cond-diag-s}.
These interpolatory conditions involve mixed terms and
do not follow the optimal bitangential Hermite interpolation formulation of the
unstructured case~\eqref{eq:lti-oc}.
But we then show that these conditions can be interpreted
as a bitangential Hermite interpolation
of not the original transfer function $H$
but of a modified multivariate transfer function;
thus showing that the classical bitangential Hermite interpolation
forms the optimality in the second-order dynamics case as well.

As in~\cite{BeaB14},
we assume that the reduced second-order model is modally damped.
Therefore, using a state space transformation,
the \ac{strom} can be brought to a form where $\hE$ and $\hK$ are real diagonal,
thus making $(\hE, \hK, \hB, \hC)$ a \ac{diagstrom}.
In particular, the matrix $\lambda^2 \hI + \lambda \hE + \hK$ is diagonal
with quadratic polynomials on its diagonal, and
thus can be decomposed into
$(\lambda \hI - \Lambda^+) (\lambda \hI - \Lambda^-)$
for some complex diagonal matrices $\Lambda^+$ and $\Lambda^-$.
This is precisely what we use in the next result.
\begin{theorem}\label{thm:so-interp}
  Let $H \in \Htwo^{\nout \times \nin}$ and
  $\Sigma_{\textnormal{D-SO}}$ be the set of second-order \acp{diagstrom}
  $(\hE, \hK, \hB, \hC)$ in~\eqref{eq:so-sys-r} such that
  $\hE, \hK \in \RRrr$ are diagonal,
  $\lambda^2 \hI + \lambda \hE + \hK$ has all of its eigenvalues in $\CC_-$,
  $\hB \in \RRri$, and
  $\hC \in \RRor$.
  Let $(\hE, \hK, \hB, \hC)$ be an $\Htwo$-optimal second-order \ac{diagstrom}
  for $H$ in $\Sigma_{\textnormal{D-SO}}$.
  Let
  $\Lambda^+ = \mydiag{\lambda_i^+}$ and
  $\Lambda^- = \mydiag{\lambda_i^-}$
  be complex diagonal matrices such that
  $\hE = -(\Lambda^+ + \Lambda^-)$ and
  $\hK = \Lambda^+ \Lambda^-$.
  Additionally,
  let all $\lambda_i^+$ and $\lambda_j^-$ be pairwise distinct.
  Define $c_i = \hC e_i$ and $b_i = \hB\tran e_i$.
  Then the transfer function $\hH$ of $(\hE, \hK, \hB, \hC)$ can be written as
  \begin{equation}\label{eq:modtf}
    \hH(s)
    = \hC \myparen*{s^2 \hI + s \hE + \hK}^{-1} \hB
    = \sum_{i = 1}^{\nrom}
    \frac{c_i b_i\tran}{\myparen*{s - \lambda_i^+} \myparen*{s - \lambda_i^-}}
  \end{equation}
  and it satisfies the interpolatory optimality conditions
  \begin{subequations}\label{eq:soc}
    \begin{align}
      \label{eq:soc1}
      \myparen*{
        H\myparen*{-\overline{\lambda_i^+}}
        - H\myparen*{-\overline{\lambda_i^-}}
      }
      b_i
       & =
      \myparen*{
        \hH\myparen*{-\overline{\lambda_i^+}}
        - \hH\myparen*{-\overline{\lambda_i^-}}
      }
      b_i, \\
      \label{eq:soc2}
      c_i\tran
      \myparen*{
        H\myparen*{-\overline{\lambda_i^+}}
        - H\myparen*{-\overline{\lambda_i^-}}
      }
       & =
      c_i\tran
      \myparen*{
        \hH\myparen*{-\overline{\lambda_i^+}}
        - \hH\myparen*{-\overline{\lambda_i^-}}
      },   \\
      \label{eq:soc3}
      c_i\tran
      H'\myparen*{-\overline{\lambda_i^+}}
      b_i
       & =
      c_i\tran
      \hH'\myparen*{-\overline{\lambda_i^+}}
      b_i, \\
      \label{eq:soc4}
      c_i\tran
      H'\myparen*{-\overline{\lambda_i^-}}
      b_i
       & =
      c_i\tran
      \hH'\myparen*{-\overline{\lambda_i^-}}
      b_i,
    \end{align}
  \end{subequations}
  for $i = 1, 2, \dots, \nrom$.
\end{theorem}
\begin{proof}
  We start by using
  $s^2 \hI + s \hE + \hK
    = (s \hI - \Lambda^+) (s \hI - \Lambda^-)$
  to obtain
  \begin{align*}
    \hH(s)
     & =
    \hC
    \myparen*{s^2 \hI + s \hE + \hK}^{-1}
    \hB
    =
    \hC
    \myparen*{s \hI - \Lambda^+}^{-1}
    \myparen*{s \hI - \Lambda^-}^{-1}
    \hB  \\
     & =
    \sum_{i = 1}^{\nrom}
    \frac{
      c_i
      b_i\tran
    }{
      \myparen*{s - \lambda_i^+}
      \myparen*{s - \lambda_i^-}
    },
  \end{align*}
  which proves~\eqref{eq:modtf}.
  We can apply \Cref{thm:cond-diag-s} with
  $\car_1(s) = s^2$,
  $\car_2(s) = s$,
  $\car_3(s) = 1$
  to recover explicit optimality conditions
  where $a_i(s) = (s - \lambda_i^+) (s - \lambda_i^-)$.
  Note that $a_i'(s) = (s - \lambda_i^+) + (s - \lambda_i^-)$ and
  $a_i''(s) = 2$.
  To simplify the notation, define
  $\kappa_i = \frac{1}{\lambda_i^+ - \lambda_i^-}$.

  From condition~\eqref{eq:cond-diag-s-1} in \Cref{thm:cond-diag-s},
  we obtain
  \begin{equation*}
    \myparen*{
      \overline{\kappa_i}
      H\myparen*{-\overline{\lambda_i^+}}
      -
      \overline{\kappa_i}
      H\myparen*{-\overline{\lambda_i^-}}
    }
    b_i
    =
    \myparen*{
      \overline{\kappa_i}
      \hH\myparen*{-\overline{\lambda_i^+}}
      -
      \overline{\kappa_i}
      \hH\myparen*{-\overline{\lambda_i^-}}
    }
    b_i.
  \end{equation*}
  Dividing by $\overline{\kappa_i}$ gives~\eqref{eq:soc1}.
  Similarly, condition~\eqref{eq:cond-diag-s-2} in \Cref{thm:cond-diag-s}
  yields~\eqref{eq:soc2}.

  Using condition~\eqref{eq:cond-diag-s-3} with $\car_3$
  (corresponding to $\hK$)
  in \Cref{thm:cond-diag-s},
  we obtain
  \begin{align*}
     &
    c_i\tran
    \myparen*{
      \overline{\kappa_i}^2
      H'\myparen*{-\overline{\lambda_i^+}}
      +
      2
      \overline{\kappa_i}^3
      H\myparen*{-\overline{\lambda_i^+}}
      +
      \overline{\kappa_i}^2
      H'\myparen*{-\overline{\lambda_i^-}}
      -
      2
      \overline{\kappa_i}^3
      H\myparen*{-\overline{\lambda_i^-}}
    }
    b_i  \\*
     & =
    c_i\tran
    \myparen*{
      \overline{\kappa_i}^2
      \hH'\myparen*{-\overline{\lambda_i^+}}
      +
      2
      \overline{\kappa_i}^3
      \hH\myparen*{-\overline{\lambda_i^+}}
      +
      \overline{\kappa_i}^2
      \hH'\myparen*{-\overline{\lambda_i^-}}
      -
      2
      \overline{\kappa_i}^3
      \hH\myparen*{-\overline{\lambda_i^-}}
    }
    b_i.
  \end{align*}
  Dividing by $\overline{\kappa_i}^2$ gives
  \begin{align*}
     &
    c_i\tran
    \myparen*{
      \myparen*{
        H'\myparen*{-\overline{\lambda_i^+}}
        + H'\myparen*{-\overline{\lambda_i^-}}
      }
      + 2
      \overline{\kappa_i}
      \myparen*{
        H\myparen*{-\overline{\lambda_i^+}}
        - H\myparen*{-\overline{\lambda_i^-}}
      }
    }
    b_i \\*
     &
    =
    c_i\tran
    \myparen*{
      \myparen*{
        \hH'\myparen*{-\overline{\lambda_i^+}}
        + \hH'\myparen*{-\overline{\lambda_i^-}}
      }
      + 2
      \overline{\kappa_i}
      \myparen*{
        \hH\myparen*{-\overline{\lambda_i^+}}
        - \hH\myparen*{-\overline{\lambda_i^-}}
      }
    }
    b_i.
  \end{align*}
  Since the terms next to $\overline{\kappa_i}$ cancel out
  using~\eqref{eq:soc1}, the last formula simplifies to
  \begin{equation}\label{eq:cond-so-bi-1}
    c_i\tran
    H'\myparen*{-\overline{\lambda_i^+}}
    b_i
    +
    c_i\tran
    H'\myparen*{-\overline{\lambda_i^-}}
    b_i
    =
    c_i\tran
    \hH'\myparen*{-\overline{\lambda_i^+}}
    b_i
    +
    c_i\tran
    \hH'\myparen*{-\overline{\lambda_i^-}}
    b_i.
  \end{equation}
  From condition~\eqref{eq:cond-diag-s-3} with $\car_2$
  (corresponding to $\hE$)
  in \Cref{thm:cond-diag-s},
  we obtain
  \begin{align*}
     &
    \overline{\lambda_i^+}
    \overline{\kappa_i}^2
    c_i\tran
    H'\myparen*{-\overline{\lambda_i^+}}
    b_i
    -
    \myparen*{
      \overline{\kappa_i}^2
      -
      2
      \overline{\lambda_i^+}
      \overline{\kappa_i}^3
    }
    c_i\tran
    H\myparen*{-\overline{\lambda_i^+}}
    b_i        \\*
     & +
    \overline{\lambda_i^-}
    \overline{\kappa_i}^2
    c_i\tran
    H'\myparen*{-\overline{\lambda_i^-}}
    b_i
    -
    \myparen*{
      \overline{\kappa_i}^2
      +
      2
      \overline{\lambda_i^-}
      \overline{\kappa_i}^3
    }
    c_i\tran
    H\myparen*{-\overline{\lambda_i^-}}
    b_i        \\*
     &
    =
    \overline{\lambda_i^+}
    \overline{\kappa_i}^2
    c_i\tran
    \hH'\myparen*{-\overline{\lambda_i^+}}
    b_i
    -
    \myparen*{
      \overline{\kappa_i}^2
      -
      2
      \overline{\lambda_i^+}
      \overline{\kappa_i}^3
    }
    c_i\tran
    \hH\myparen*{-\overline{\lambda_i^+}}
    b_i        \\*
     & \quad +
    \overline{\lambda_i^-}
    \overline{\kappa_i}^2
    c_i\tran
    \hH'\myparen*{-\overline{\lambda_i^-}}
    b_i
    -
    \myparen*{
      \overline{\kappa_i}^2
      +
      2
      \overline{\lambda_i^-}
      \overline{\kappa_i}^3
    }
    c_i\tran
    \hH\myparen*{-\overline{\lambda_i^-}}
    b_i.
  \end{align*}
  Dividing by $\overline{\kappa_i}^3$ gives
  \begin{align*}
     &
    \myparen*{\overline{\lambda_i^+} - \overline{\lambda_i^-}}
    \overline{\lambda_i^+}
    c_i\tran
    H'\myparen*{-\overline{\lambda_i^+}}
    b_i
    +
    \myparen*{\overline{\lambda_i^+} + \overline{\lambda_i^-}}
    c_i\tran
    H\myparen*{-\overline{\lambda_i^+}}
    b_i        \\*
     & +
    \myparen*{\overline{\lambda_i^+} - \overline{\lambda_i^-}}
    \overline{\lambda_i^-}
    c_i\tran
    H'\myparen*{-\overline{\lambda_i^-}}
    b_i
    -
    \myparen*{\overline{\lambda_i^+} + \overline{\lambda_i^-}}
    c_i\tran
    H\myparen*{-\overline{\lambda_i^-}}
    b_i        \\*
     & =
    \myparen*{\overline{\lambda_i^+} - \overline{\lambda_i^-}}
    \overline{\lambda_i^+}
    c_i\tran
    \hH'\myparen*{-\overline{\lambda_i^+}}
    b_i
    +
    \myparen*{\overline{\lambda_i^+} + \overline{\lambda_i^-}}
    c_i\tran
    \hH\myparen*{-\overline{\lambda_i^+}}
    b_i        \\*
     & \quad +
    \myparen*{\overline{\lambda_i^+} - \overline{\lambda_i^-}}
    \overline{\lambda_i^-}
    c_i\tran
    \hH'\myparen*{-\overline{\lambda_i^-}}
    b_i
    -
    \myparen*{\overline{\lambda_i^+} + \overline{\lambda_i^-}}
    c_i\tran
    \hH\myparen*{-\overline{\lambda_i^-}}
    b_i.
  \end{align*}
  Since the terms next to $\overline{\lambda_i^+} + \overline{\lambda_i^-}$
  cancel out using~\eqref{eq:soc1}, this simplifies to
  \begin{equation}\label{eq:cond-so-bi-2}
    \overline{\lambda_i^+}
    c_i\tran
    H'\myparen*{-\overline{\lambda_i^+}}
    b_i
    +
    \overline{\lambda_i^-}
    c_i\tran
    H'\myparen*{-\overline{\lambda_i^-}}
    b_i
    =
    \overline{\lambda_i^+}
    c_i\tran
    \hH'\myparen*{-\overline{\lambda_i^+}}
    b_i
    +
    \overline{\lambda_i^-}
    c_i\tran
    \hH'\myparen*{-\overline{\lambda_i^-}}
    b_i.
  \end{equation}
  Note that~\eqref{eq:cond-so-bi-1} and~\eqref{eq:cond-so-bi-2}
  can be merged together to give
  \begin{equation*}
    \renewcommand{\arraystretch}{1.5}
    \begin{bmatrix}
      1                      & 1                      \\
      \overline{\lambda_i^+} & \overline{\lambda_i^-}
    \end{bmatrix}
    \begin{bmatrix}
      c_i\tran H'\myparen*{-\overline{\lambda_i^+}} b_i \\
      c_i\tran H'\myparen*{-\overline{\lambda_i^-}} b_i
    \end{bmatrix}
    =
    \begin{bmatrix}
      1                      & 1                      \\
      \overline{\lambda_i^+} & \overline{\lambda_i^-}
    \end{bmatrix}
    \begin{bmatrix}
      c_i\tran \hH'\myparen*{-\overline{\lambda_i^+}} b_i \\
      c_i\tran \hH'\myparen*{-\overline{\lambda_i^-}} b_i
    \end{bmatrix}.
  \end{equation*}
  Since we assumed that $\lambda_i^+ \neq \lambda_i^-$,
  it follows that
  \begin{equation*}
    c_i\tran H'\myparen*{-\overline{\lambda_i^+}} b_i
    = c_i\tran \hH'\myparen*{-\overline{\lambda_i^+}} b_i
    \quad \text{and} \quad
    c_i\tran H'\myparen*{-\overline{\lambda_i^-}} b_i
    = c_i\tran \hH'\myparen*{-\overline{\lambda_i^-}} b_i,
  \end{equation*}
  proving~\eqref{eq:soc3} and~\eqref{eq:soc4}, and completing the proof.
\end{proof}
Therefore, our general framework for optimality
as given in \Cref{thm:cond-diag-s} and \Cref{cor:cond}
recovers the interpolatory optimality conditions from~\cite{BeaB14} as a special
case.

At a first glance,
the interpolatory optimality conditions~\eqref{eq:soc}
for second-order structures are different from those
of the classical bitangential Hermite interpolation conditions
of the unstructured $\Htwo$ approximation problem~\eqref{eq:lti-oc}.
More specifically, even though the bitangential Hermite
conditions~\eqref{eq:soc3}--\eqref{eq:soc4} resemble the classical case,
the left- and right-tangential Lagrange
conditions~\eqref{eq:soc1}--\eqref{eq:soc2} are rather different since they
impose interpolating a difference as opposed to the original transfer function.
However, these new interpolatory conditions can still be interpreted
as bitangential Hermite conditions for a modified transfer function
as we show next.
\begin{corollary}\label{cor:sobh}
  Let the assumptions in \Cref{thm:so-interp} hold.
  Define the 2D full-order and reduced-order transfer functions
  \begin{equation}\label{eq:Gs1s2}
    G(s_1, s_2) = H(s_1) - H(s_2)
    \quad \text{and} \quad
    \hG(s_1, s_2) = \hH(s_1) - \hH(s_2).
  \end{equation}
  Then, the optimality conditions~\eqref{eq:soc1}--\eqref{eq:soc4} are,
  respectively, equivalent to
  \begin{subequations}\label{eq:sobh}
    \begin{align}
      \label{eq:sobh1}
      G\myparen*{-\overline{\lambda_i^+}, -\overline{\lambda_i^-}}
      b_i
       & =
      \hG\myparen*{-\overline{\lambda_i^+}, -\overline{\lambda_i^-}}
      b_i,                                                            \\
      \label{eq:sobh2}
      c_i\tran
      G\myparen*{-\overline{\lambda_i^+}, -\overline{\lambda_i^-}}
       & =
      c_i\tran
      \hG\myparen*{-\overline{\lambda_i^+}, -\overline{\lambda_i^-}}, \\
      \label{eq:sobh3}
      c_i\tran
      \frac{\partial G}{\partial s_1}
      \myparen*{-\overline{\lambda_i^+}, -\overline{\lambda_i^-}}
      b_i
       & =
      c_i\tran
      \frac{\partial \hG}{\partial s_1}
      \myparen*{-\overline{\lambda_i^+}, -\overline{\lambda_i^-}}
      b_i,                                                            \\
      \label{eq:sobh4}
      c_i\tran
      \frac{\partial G}{\partial s_2}
      \myparen*{-\overline{\lambda_i^+}, -\overline{\lambda_i^-}}
      b_i
       & =
      c_i\tran
      \frac{\partial \hG}{\partial s_2}
      \myparen*{-\overline{\lambda_i^+}, -\overline{\lambda_i^-}}
      b_i,
    \end{align}
  \end{subequations}
  for $i = 1, 2, \dots, \nrom$.
\end{corollary}
\begin{proof}
  The condition~\eqref{eq:sobh1} and~\eqref{eq:sobh2} directly follow from,
  respectively,~\eqref{eq:soc1} and~\eqref{eq:soc2}
  based on the definitions of $G$ and $\hG$ in~\eqref{eq:Gs1s2}.
  Also note that
  $\frac{\partial G}{\partial s_1}(s_1, s_2) = H'(s_1)$ and
  $\frac{\partial G}{\partial s_2}(s_1, s_2) = -H'(s_2)$; and
  similarly for $\hG$.
  These observations immediately reveal
  that~\eqref{eq:sobh3} and~\eqref{eq:sobh4} are equivalent to,
  respectively,~\eqref{eq:soc3} and~\eqref{eq:soc4}.
\end{proof}
Therefore, as in the unstructured case,
$\Htwo$-optimal reduced-order modeling of second-order systems
requires bitangential Hermite interpolation.
The optimal interpolation points are still the mirror images
of the reduced-order poles and
the tangential directions result from the residues of the \ac{rom}.
However, what needs to be interpolated is a modified 2D transfer function $G$.
In our earlier work~\cite{MliG23b}, we showed that
in addition to the classical $\Htwo$-optimal approximation problem,
the bitangential Hermite interpolation
formed the necessary conditions for optimality
for the rational nonlinear least-squares fitting of \ac{lti} systems and
for the $\Ltwo$-optimal approximation of stationary problems.
In both cases, the interpolation had to be enforced
on a modified mapping rather than the original one.
\Cref{cor:sobh} proves this to be the case
for structure-preserving $\Htwo$-optimal approximation
of second-order \ac{lti} systems as well.
Thus, bitangential Hermite interpolation remains
the unifying framework even for a larger class of problems.

\subsection{Numerical Example}

To demonstrate the results of \Cref{thm:so-interp},
we use the following second-order system as the \ac{fom}:
\begin{gather*}
  E = \mydiag*{\frac{3}{10}, \frac{2}{10}, 3, 2}, \
  K = \mydiag*{\frac{2}{100}, \frac{101}{100}, 2, 2}, \\
  B =
  \begin{bmatrix}
    1 & 1 & 1 \\
    1 & 2 & 2 \\
    1 & 3 & 1 \\
    1 & 4 & 2
  \end{bmatrix}\!, \
  C =
  \begin{bmatrix}
    1 & 1 & 1 & 1 \\
    1 & 2 & 3 & 4 \\
  \end{bmatrix}\!,
\end{gather*}
which was chosen to have two pairs (one real, one complex) of poles close to the
imaginary axis and two pairs that are further away:
\[
  \frac{1}{\myparen*{s + \frac{1}{10}} \myparen*{s + \frac{2}{10}}}, \
  \frac{1}{\myparen*{s + \frac{1}{10}}^2 + 1}, \
  \frac{1}{\myparen*{s + 1} \myparen*{s + 2}}, \
  \frac{1}{\myparen*{s + 1}^2 + 1}.
\]
We then find \iac{diagstrom} of order $\nrom = 2$ using the Nelder-Mead
method~\cite{NelM65} initialized with
\[
  \hE = I_{4, 2}\tran E I_{4, 2}, \
  \hK = I_{4, 2}\tran K I_{4, 2}, \
  \hB = I_{4, 2}\tran B, \
  \hC = C I_{4, 2},
\]
where $I_{4, 2} \in \RR^{4 \times 2}$ are the first two columns of $I_4$.
The relative $\Htwo$ error of the initial \ac{diagstrom} is $0.32393$
(rounded to $5$ significant digits).
In the Nelder-Mead minimization, the \ac{diagstrom} is parameterized using a
vector $x \in \RR^{14}$ as
\begin{gather*}
  \hE = \mydiag{x_1, x_2}, \
  \hK = \mydiag{x_3, x_4}, \
  \hB =
  \begin{bmatrix}
    x_5 & x_6 & x_7    \\
    x_8 & x_9 & x_{10}
  \end{bmatrix}\!, \
  \hC =
  \begin{bmatrix}
    x_{11} & x_{12} \\
    x_{13} & x_{14}
  \end{bmatrix}\!.
\end{gather*}
The resulting optimal model is a \ac{diagstrom} with
\begin{gather*}
  \hE = \mydiag{0.32897, 0.5025}, \
  \hK = \mydiag{0.021932, 1.0745}, \\
  \hB =
  \begin{bmatrix}
    0.46508 & 0.4814 & 0.46506 \\
    1.024   & 2.7071 & 1.919
  \end{bmatrix}\!, \
  \hC =
  \begin{bmatrix}
    2.3466 & 1.8504 \\
    2.416  & 4.7526
  \end{bmatrix}\!,
\end{gather*}
and relative $\Htwo$ error of $0.15692$.
Its poles are
\[
  \lambda_1^+ = -0.092907, \
  \lambda_1^- = -0.23607, \
  \lambda_2^{\pm} = -0.25125 \pm 1.0057 \imag.
\]
To check whether the resulting \ac{diagstrom} satisfies the structured
optimality conditions, we define the relative distance function as
$\mathrm{reldist}(x, y) = \frac{\norm{x - y}}{\norm{x}}$.
We find that
\begin{gather*}
  \mathrm{reldist}\myparen*{
    \myparen*{
      H\myparen*{-\overline{\lambda_1^+}}
      - H\myparen*{-\overline{\lambda_1^-}}
    }
    b_1,
    \myparen*{
      \hH\myparen*{-\overline{\lambda_1^+}}
      - \hH\myparen*{-\overline{\lambda_1^-}}
    }
    b_1
  }
  \approx 1.3 \times 10^{-8}, \\
  \mathrm{reldist}\myparen*{
    \myparen*{
      H\myparen*{-\overline{\lambda_2^+}}
      - H\myparen*{-\overline{\lambda_2^-}}
    }
    b_2,
    \myparen*{
      \hH\myparen*{-\overline{\lambda_2^+}}
      - \hH\myparen*{-\overline{\lambda_2^-}}
    }
    b_2
  }
  \approx 4.2 \times 10^{-8}, \\
  \mathrm{reldist}\myparen*{
    c_1\tran
    \myparen*{
      H\myparen*{-\overline{\lambda_1^+}}
      - H\myparen*{-\overline{\lambda_1^-}}
    },
    c_1\tran
    \myparen*{
      \hH\myparen*{-\overline{\lambda_1^+}}
      - \hH\myparen*{-\overline{\lambda_1^-}}
    }
  }
  \approx 1.1 \times 10^{-8}, \\
  \mathrm{reldist}\myparen*{
    c_2\tran
    \myparen*{
      H\myparen*{-\overline{\lambda_2^+}}
      - H\myparen*{-\overline{\lambda_2^-}}
    },
    c_2\tran
    \myparen*{
      \hH\myparen*{-\overline{\lambda_2^+}}
      - \hH\myparen*{-\overline{\lambda_2^-}}
    }
  }
  \approx 3.6 \times 10^{-8}, \\
  \mathrm{reldist}\myparen*{
    c_1\tran
    H'\myparen*{-\overline{\lambda_1^+}}
    b_1,
    c_1\tran
    \hH'\myparen*{-\overline{\lambda_1^+}}
    b_1
  }
  \approx 1.4 \times 10^{-8}, \\
  \mathrm{reldist}\myparen*{
    c_1\tran
    H'\myparen*{-\overline{\lambda_1^-}}
    b_1,
    c_1\tran
    \hH'\myparen*{-\overline{\lambda_1^-}}
    b_1
  }
  \approx 4.6 \times 10^{-9}, \\
  \mathrm{reldist}\myparen*{
    c_2\tran
    H'\myparen*{-\overline{\lambda_2^{\pm}}}
    b_2,
    c_2\tran
    \hH'\myparen*{-\overline{\lambda_2^{\pm}}}
    b_2
  }
  \approx 4.6 \times 10^{-8},
\end{gather*}
which shows that the optimality conditions of \cref{thm:so-interp} are satisfied
(within the numerical accuracy of the optimization algorithm).
We also check the unstructured optimality conditions and find that
\begin{gather*}
  \mathrm{reldist}\myparen*{
    H\myparen*{-\overline{\lambda_1^+}}
    b_1,
    \hH\myparen*{-\overline{\lambda_1^+}}
    b_1
  }
  \approx 2.5 \times 10^{-2}, \\
  \mathrm{reldist}\myparen*{
    H\myparen*{-\overline{\lambda_1^-}}
    b_1,
    \hH\myparen*{-\overline{\lambda_1^-}}
    b_1
  }
  \approx 4.7 \times 10^{-2}, \\
  \mathrm{reldist}\myparen*{
    H\myparen*{-\overline{\lambda_2^{\pm}}}
    b_2,
    \hH\myparen*{-\overline{\lambda_2^{\pm}}}
    b_2
  }
  \approx 1.0 \times 10^{-1}, \\
  \mathrm{reldist}\myparen*{
    c_1\tran
    H\myparen*{-\overline{\lambda_1^+}},
    c_1\tran
    \hH\myparen*{-\overline{\lambda_1^+}}
  }
  \approx 2.7 \times 10^{-2}, \\
  \mathrm{reldist}\myparen*{
    c_1\tran
    H\myparen*{-\overline{\lambda_1^-}},
    c_1\tran
    \hH\myparen*{-\overline{\lambda_1^-}}
  }
  \approx 5.1 \times 10^{-2}, \\
  \mathrm{reldist}\myparen*{
    c_2\tran
    H\myparen*{-\overline{\lambda_2^{\pm}}},
    c_2\tran
    \hH\myparen*{-\overline{\lambda_2^{\pm}}}
  }
  \approx 1.1 \times 10^{-1},
\end{gather*}
which demonstrates that the \ac{diagstrom} does not satisfy the unstructured
interpolatory necessary optimality conditions~\eqref{eq:lti-oc}.
The code for this numerical example is available at~\cite{Mli24}.

\section{Port-Hamiltonian Systems}%
\label{sec:ph}

\Ac{lti} \ac{ph} systems naturally arise in modeling a wide range of
physical, engineering, and biological systems.
They generalize the classical Hamiltonian structure
by including inputs and outputs,
thus allowing interaction with the environment via their \emph{ports}.
The \ac{ph} systems are important in energy-based modeling
as their form ensures passivity and allows energy-preserving interconnections.
We refer the reader to, e.g.,~\cite{VanJ14,MehU22,BirZ12},
for more details on \ac{ph} systems.

Thus, given an \ac{lti} system $H$, in this section we consider \acp{strom}
that have the \ac{lti} \ac{ph} form
\begin{subequations}\label{eq:ph-rom}
  \begin{align}
    \dot{\hx}(t)
     & =
    \myparen*{\hJ - \hR} \hx(t)
    + \hB u(t), \\*
    \hy(t)
     & =
    \hB\tran \hx(t),
  \end{align}
\end{subequations}
where
$\hJ, \hR \in \RRrr$,
$\hB \in \RRri$,
$\hJ\tran = -\hJ$, and
$\hR = \hR\tran \succcurlyeq 0$%
\footnote{There are slightly more general forms of the \ac{ph} systems.
  For the conciseness of the presentation,
  we focus on the form~\eqref{eq:ph-rom}.}.
Additionally, we assume $\hJ - \hR$ is asymptotically stable,
which is guaranteed when $\hR \succ 0$.

Given an \ac{lti} system $H$,
our goal is to find \iac{ph} \ac{rom} as in~\eqref{eq:ph-rom}
that minimizes the $\Htwo$ error~\eqref{eq:h2-error}.
Several approaches have been proposed for this problem.
In~\cite{GugPB+12},
the authors propose an iterative algorithm similar to \ac{irka},
called IRKA-PH,
which finds \iac{rom} that, upon convergence,
satisfies one, namely~\eqref{eq:lti-oc1}, out of the three necessary conditions
for \emph{unstructured} first-order systems~\eqref{eq:lti-oc}.
However, it is important to point out that the condition~\eqref{eq:lti-oc1}
that IRKA-PH satisfies does not correspond to true optimality conditions for
\iac{ph} system;
it simply works with the conditions for the unstructured case.
Deriving the true structured optimality conditions for \ac{ph} systems
is the main goal of this section.
Exploiting the minimal solution of an algebraic Riccati equation
related to \ac{ph} systems,
Breiten and Unger~\cite{BreU22} significantly improved the performance of
IRKA-PH\@.
Optimal-$\Htwo$ model reduction with \ac{ph} structure was also considered,
e.g., in~\cite{MosL20,SMMV22,MSMV22} where gradient-based optimization with
structure preservation is used to construct the \acp{rom}.

\subsection{Interpolatory Conditions}

These aforementioned methods provide high-quality \acp{rom} with \ac{ph}
structure,
yet they do not derive or work with interpolatory optimality conditions,
which is our main focus here.
Interpolatory $\Htwo$ conditions with \ac{ph} structure were first studied
in~\cite{BeaB14} where the authors derived (a subset of) the
necessary conditions for $\Htwo$-optimality,
shown in~\eqref{eq:BeaB14-pH} below.
We first show in \Cref{thm:BeaB14-pH-interp} that these conditions follow from
our general framework, \Cref{cor:cond}, directly.
But more importantly what this derivation will reveal is
that the interpolatory conditions in~\cite{BeaB14} are not complete
in the sense that they correspond to only~\eqref{eq:cond-A} in \Cref{cor:cond}
(corresponding to the gradient with respect to $\hJ - \hR$)
without the other necessary conditions.
Then based on this observation, in \Cref{thm:pH-interp-tangential},
we provide the remaining (realization-dependent) conditions for
$\Htwo$-optimal model reduction with \ac{ph} structure.
With an additional structural assumption,
in \Cref{thm:pH-interp},
we provide a realization-independent set of interpolatory optimality conditions.
\begin{theorem}\label{thm:BeaB14-pH-interp}
  Let $H \in \Htwo^{\nin \times \nin}$ and
  $\Sigma_{\textnormal{pH}}$ be the set of \ac{ph} \acp{strom}
  $(\hJ, \hR, \hB)$ in~\eqref{eq:ph-rom} such that
  $\hJ, \hR \in \RRrr$,
  $\hJ\tran = -\hJ$,
  $\hR = \hR\tran \succcurlyeq 0$,
  $\hJ - \hR$ has all of its eigenvalues in $\CC_-$, and
  $\hB \in \RRri$.
  Let $(\hJ, \hR, \hB)$ be an $\Htwo$-optimal \ac{ph} \ac{strom} for $H$ in
  $\Sigma_{\textnormal{pH}}$.
  Suppose that $\hR \succ 0$ and that the transfer function $\hH$ of
  $(\hJ, \hR, \hB)$ has $\nrom$ distinct poles and is represented as
  $\hH(s) = \sum_{i = 1}^{\nrom} \frac{c_i b_i\herm}{s - \lambda_i}$.
  Then
  \begin{subequations}\label{eq:BeaB14-pH}
    \begin{align}
      \label{eq:BeaB14-pH-1}
      c_i\herm
      \myparen*{
        H\myparen*{-\overline{\lambda_i}}
        - H\myparen*{-\overline{\lambda_j}}
      }
      b_j
       & =
      c_i\herm
      \myparen*{
        \hH\myparen*{-\overline{\lambda_i}}
        - \hH\myparen*{-\overline{\lambda_j}}
      }
      b_j, \\
      \label{eq:BeaB14-pH-2}
      c_i\herm
      H'\myparen*{-\overline{\lambda_i}}
      b_i
       & =
      c_i\herm
      \hH'\myparen*{-\overline{\lambda_i}}
      b_i,
    \end{align}
  \end{subequations}
  for $i, j = 1, 2, \dots, \nrom$.
\end{theorem}
\begin{proof}
  The matrix $\hJ - \hR$ satisfies
  $(\hJ - \hR) + (\hJ - \hR)\tran = -2 \hR \prec 0$.
  Conversely, any matrix $\hA$ such that $\hA + \hA\tran \prec 0$
  can be decomposed as $\hA = \hJ - \hR$
  where $\hJ$ is skew-symmetric and
  $\hR$ is symmetric positive definite.
  Since the set of matrices
  $S = \{\hA \in \RRrr : \hA + \hA\tran \prec 0\}$ is open in $\RRrr$ and
  $\hJ - \hR \in S$,
  the condition~\eqref{eq:cond-A} corresponding to the gradient
  with respect to $\hJ - \hR$ holds, i.e.,
  \begin{equation} \label{eq:equ1}
    \int_{-\infty}^{\infty}
    \myparen*{\imag \omega \hI - \hJ + \hR}\mherm
    \hB
    \myparen*{
      H(\imag \omega)
      - \hH(\imag \omega)
    }
    \hB\tran
    \myparen*{\imag \omega \hI - \hJ + \hR}\mherm
    \dif{\omega}
    =
    0.
  \end{equation}
  Let $\hT \in \CCrr$ be an invertible matrix such that
  $\hT^{-1} (\hJ - \hR) \hT = \Lambda$.
  Then, after multiplying from the left by $\hT\herm$ and
  from the right by $\hT\mherm$,~\eqref{eq:equ1} becomes
  \begin{equation} \label{eq:equ2}
    \int_{-\infty}^{\infty}
    \myparen*{\imag \omega \hI - \Lambda}\mherm
    \hT\herm
    \hB
    \myparen*{
      H(\imag \omega)
      - \hH(\imag \omega)
    }
    \hB\tran
    \hT\mherm
    \myparen*{\imag \omega \hI - \Lambda}\mherm
    \dif{\omega}
    =
    0.
  \end{equation}
  Note that $b_j = \hB\tran \hT\mherm e_j$ and
  $c_i\herm = e_i\tran \hT\herm \hB$.
  Multiplying~\eqref{eq:equ2} from the left by $e_i\tran$ and
  from the right by $e_j$ gives
  \begin{equation*}
    \int_{-\infty}^{\infty}
    \frac{
      c_i\herm
      \myparen*{
        H(\imag \omega)
        - \hH(\imag \omega)
      }
      b_j
    }{
      \myparen*{-\imag \omega - \overline{\lambda_i}}
      \myparen*{-\imag \omega - \overline{\lambda_j}}
    }
    \dif{\omega}
    =
    0.
  \end{equation*}
  Substituting $s = \imag \omega$ in this last equality yields
  \begin{equation}\label{eq:equ3}
    \oint_{\imag \RR}
    \frac{
      c_i\herm
      \myparen*{
        H(s)
        - \hH(s)
      }
      b_j
    }{
      \myparen*{s + \overline{\lambda_i}}
      \myparen*{s + \overline{\lambda_j}}
    }
    \dif{s}
    =
    0.
  \end{equation}
  When $i \neq j$,~\eqref{eq:equ3} immediately leads to the conditions
  in~\eqref{eq:BeaB14-pH-1}.
  And finally, for $i = j$, we obtain~\eqref{eq:BeaB14-pH-2}.
\end{proof}
Thus, we are able to recover the interpolatory conditions from~\cite{BeaB14} as
a special case of our general framework.
Additionally, and more importantly,
the proof reveals that in the derivation of these conditions
only the information in $\hJ - \hR$ is used,
but not in $\hB$.
Thus, these conditions do not provide the full set of interpolatory conditions
for optimality.
To illustrate this, for simplicity, assume we have a SISO \ac{lti} system.
Even though~\eqref{eq:BeaB14-pH} seems to correspond to $\nrom^2 + \nrom$
conditions,
that is not the case.
If the interpolatory conditions
\[
  H\myparen*{-\overline{\lambda_i}}
  - H\myparen*{-\overline{\lambda_j}}
  =
  \hH\myparen*{-\overline{\lambda_i}}
  - \hH\myparen*{-\overline{\lambda_j}}
\]
and
\[
  H\myparen*{-\overline{\lambda_j}}
  - H\myparen*{-\overline{\lambda_k}}
  =
  \hH\myparen*{-\overline{\lambda_j}}
  - \hH\myparen*{-\overline{\lambda_k}}
\]
hold, then by adding them, we automatically obtain
\[
  H\myparen*{-\overline{\lambda_i}}
  - H\myparen*{-\overline{\lambda_k}}
  =
  \hH\myparen*{-\overline{\lambda_i}}
  - \hH\myparen*{-\overline{\lambda_k}}.
\]
Then, for the SISO case, one has a total of $2 \nrom - 1$ conditions,
which is not enough to uniquely specify a rational function of order~$\nrom$.
The following theorem provides the remaining conditions.
\begin{theorem}\label{thm:pH-interp-tangential}
  Let the assumptions in \Cref{thm:BeaB14-pH-interp} hold.
  Furthermore, let $\hT$ be an eigenvector matrix of $\hJ - \hR$ and
  denote
  $\hT e_i = t_i$,
  $\hT\mherm e_i = s_i$,
  $\hB\tran t_i = c_i$, and,
  $\hB\tran s_i = b_i$.
  Then, in addition to~\eqref{eq:BeaB14-pH}, $\hH$ also satisfies
  \begin{equation}\label{eq:pH-interp-tangential}
    \sum_{i = 1}^{\nrom}
    \myparen*{
      H\myparen*{-\overline{\lambda_i}}
      b_i
      t_i\herm
      +
      H\myparen*{-\overline{\lambda_i}}\herm
      c_i
      s_i\herm
    }
    =
    \sum_{i = 1}^{\nrom}
    \myparen*{
      \hH\myparen*{-\overline{\lambda_i}}
      b_i
      t_i\herm
      +
      \hH\myparen*{-\overline{\lambda_i}}\herm
      c_i
      s_i\herm
    }.
  \end{equation}
\end{theorem}
\begin{proof}
  The gradient with respect to $\hB$ follows from \Cref{thm:grad}.
  In particular, since $\hB$ in the \ac{ph} system~\eqref{eq:ph-rom}
  appears in both $\Br$ and $\Cr$,
  the gradient of the squared $\Htwo$ error with respect to $\hB$ is
  \begin{align*}
    \nabla_{\overline{\hB}} \obj
     & =
    \int_{-\infty}^{\infty}
    \myparen*{\imag \omega \hI - \hJ + \hR}\mherm
    \hB
    \myparen*{
      \hH(\imag \omega)
      - H(\imag \omega)
    }
    \dif{\omega} \\*
     & \quad
    +
    \myparen*{
      \int_{-\infty}^{\infty}
      \myparen*{
        \hH(\imag \omega)
        - H(\imag \omega)
      }
      \hB\tran
      \myparen*{\imag \omega \hI - \hJ + \hR}\mherm
      \dif{\omega}
    }\herm       \\
     & =
    \int_{-\infty}^{\infty}
    \myparen*{\imag \omega \hI - \hJ + \hR}\mherm
    \hB
    \myparen*{
      \hH(\imag \omega)
      - H(\imag \omega)
    }
    \dif{\omega} \\*
     & \quad
    +
    \int_{-\infty}^{\infty}
    \myparen*{\imag \omega \hI - \hJ + \hR}^{-1}
    \hB
    \myparen*{
      \hH(\imag \omega)\herm
      - H(\imag \omega)\herm
    }
    \dif{\omega}.
  \end{align*}
  Using $\hT^{-1} (\hJ - \hR) \hT = \Lambda$, the gradient becomes
  \begin{align*}
    \nabla_{\overline{\hB}} \obj
     & =
    \hT\mherm
    \int_{-\infty}^{\infty}
    \myparen*{\imag \omega \hI - \Lambda}\mherm
    \hT\herm
    \hB
    \myparen*{
      \hH(\imag \omega)
      - H(\imag \omega)
    }
    \dif{\omega} \\*
     & \quad
    +
    \hT
    \int_{-\infty}^{\infty}
    \myparen*{\imag \omega \hI - \Lambda}^{-1}
    \hT^{-1}
    \hB
    \myparen*{
      \hH(\imag \omega)\herm
      - H(\imag \omega)\herm
    }
    \dif{\omega}.
  \end{align*}
  Using $\hI = \sum_{i = 1}^{\nrom} e_i e_i\tran$, we find
  \begin{align*}
    \nabla_{\overline{\hB}} \obj
     & =
    \hT\mherm
    \sum_{i = 1}^{\nrom} e_i e_i\tran
    \int_{-\infty}^{\infty}
    \myparen*{\imag \omega \hI - \Lambda}\mherm
    \hT\herm
    \hB
    \myparen*{
      \hH(\imag \omega)
      - H(\imag \omega)
    }
    \dif{\omega} \\*
     & \quad
    +
    \hT
    \sum_{i = 1}^{\nrom} e_i e_i\tran
    \int_{-\infty}^{\infty}
    \myparen*{\imag \omega \hI - \Lambda}^{-1}
    \hT^{-1}
    \hB
    \myparen*{
      \hH(\imag \omega)\herm
      - H(\imag \omega)\herm
    }
    \dif{\omega},
  \end{align*}
  which simplifies to
  \begin{align*}
    \nabla_{\overline{\hB}} \obj
     & =
    \sum_{i = 1}^{\nrom}
    s_i
    c_i\herm
    \int_{-\infty}^{\infty}
    \frac{
      \hH(\imag \omega)
      - H(\imag \omega)
    }{
      -\imag \omega - \overline{\lambda_i}
    }
    \dif{\omega}
    +
    \sum_{i = 1}^{\nrom}
    t_i
    b_i\herm
    \overline{
      \int_{-\infty}^{\infty}
      \frac{
        \hH(\imag \omega)\tran
        - H(\imag \omega)\tran
      }{
        -\imag \omega - \overline{\lambda_i}
      }
      \dif{\omega}
    }.
  \end{align*}
  Using the substitution $s = \imag \omega$ and the Cauchy integral formula,
  and equating the gradient to zero,
  we obtain
  \begin{align*}
    0
     & =
    \sum_{i = 1}^{\nrom}
    s_i
    c_i\herm
    \myparen*{
      \hH\myparen*{-\overline{\lambda_i}}
      - H\myparen*{-\overline{\lambda_i}}
    }
    +
    \sum_{i = 1}^{\nrom}
    t_i
    b_i\herm
    \myparen*{
      \hH\myparen*{-\overline{\lambda_i}}\herm
      - H\myparen*{-\overline{\lambda_i}}\herm
    }.
  \end{align*}
  Applying conjugate transpose gives the tangential interpolatory
  condition~\eqref{eq:pH-interp-tangential}.
\end{proof}
Now, the optimality conditions~\eqref{eq:BeaB14-pH} in
\Cref{thm:BeaB14-pH-interp} together with the new optimality
condition~\eqref{eq:pH-interp-tangential} in \Cref{thm:pH-interp-tangential}
describe \emph{the full set of interpolatory conditions} for $\Htwo$-optimal
approximation of \ac{ph} systems.

However, note the major difference in this new additional
condition~\eqref{eq:pH-interp-tangential}.
In the earlier condition~\eqref{eq:BeaB14-pH}
(and the other conditions derived in the earlier sections),
interpolation is realization-independent, i.e.,
it does not depend on a specific state-space form and
only uses realization-independent quantities such as poles and residues.
However, the new condition~\eqref{eq:pH-interp-tangential} is
realization-dependent due to its dependence on the eigenvectors of $\hJ - \hR$.

To remove this dependence on eigenvectors,
we require an additional assumption,
namely the normality of $\hJ - \hR$.
Then there exists a unitary matrix $\hT \in \CCrr$ such that
$\hT\herm (\hJ - \hR) \hT$ is diagonal.
Additionally, $\hT\herm \hJ \hT$ and $\hT\herm \hR \hT$ are both diagonal as
they are the skew-Hermitian and negative Hermitian parts of
$\hT\herm (\hJ - \hR) \hT$.
Thus, $(\hT\herm \hJ \hT, \hT\herm \hR \hT, \hT\herm \hB)$ is also \iac{ph}
system, in particular, \iac{ph} \ac{diagstrom}.
We then obtain a full set of realization-independent interpolatory optimality
conditions.
\begin{theorem}\label{thm:pH-interp}
  Let $H \in \Htwo^{\nin \times \nin}$ and
  $\Sigma_{\textnormal{D-pH}}$ be the set of \ac{ph} \acp{diagstrom}
  $(\hJ, \hR, \hB)$ in~\eqref{eq:ph-rom} such that
  $\hJ, \hR \in \CCrr$ are diagonal,
  $\hJ = -\hJ\herm$,
  $\hR = \hR\herm \succ 0$, and
  $\hB \in \CCri$.
  Let $(\hJ, \hR, \hB)$ be an $\Htwo$-optimal \ac{ph} \ac{diagstrom} for $H$ in
  $\Sigma_{\textnormal{D-pH}}$.
  Furthermore, let the transfer function $\hH$ of $(\hJ, \hR, \hB)$ have $\nrom$
  distinct poles.
  Then, with $\hJ - \hR = \mydiag{\lambda_i}$ and $b_i\herm = e_i\tran \hB$,
  \begin{equation}\label{eq:pH-pole-res-normal}
    \hH(s) = \sum_{i = 1}^{\nrom} \frac{b_i b_i\herm}{s - \lambda_i}
  \end{equation}
  and it satisfies the interpolatory optimality conditions
  \begin{subequations}\label{eq:new-ph}
    \begin{align}
      \label{eq:new-ph-1}
      \myparen*{
        H\myparen*{-\overline{\lambda_i}}
        + H\myparen*{-\overline{\lambda_i}}\herm
      }
      b_i
       & =
      \myparen*{
        \hH\myparen*{-\overline{\lambda_i}}
        + \hH\myparen*{-\overline{\lambda_i}}\herm
      }
      b_i, \\
      \label{eq:new-ph-2}
      b_i\herm
      H'\myparen*{-\overline{\lambda_i}}
      b_i
       & =
      b_i\herm
      \hH'\myparen*{-\overline{\lambda_i}}
      b_i,
    \end{align}
  \end{subequations}
  for $i = 1, 2, \dots, \nrom$.
\end{theorem}
\begin{proof}
  The pole-residue form~\eqref{eq:pH-pole-res-normal} follows directly.
  The tangential Lagrange condition~\eqref{eq:new-ph-1} follows
  from~\eqref{eq:pH-interp-tangential},
  noting that $c_i = b_i$ and $t_i = s_i = e_i$ since $\hJ - \hR$ is diagonal.
  The bitangential Hermite condition~\eqref{eq:new-ph-2} follows
  from~\eqref{eq:BeaB14-pH-2} in \Cref{thm:BeaB14-pH-interp} using
  \Cref{lem:diag} and the fact that $\hJ - \hR$ is diagonal.
\end{proof}
Therefore, using our general framework for $\Ltwo$-optimality and
with the additional assumption of normality,
we obtain more familiar bitangential Hermite interpolation conditions for
\ac{ph} systems,
better mimicking the unstructured case.
Moreover, we obtain a full set of conditions.
Similar to second-order systems,
we can reformulate these conditions as true bitangential Hermite conditions
for a modified transfer function as we show next.
\begin{corollary}
  Let the assumptions in \Cref{thm:pH-interp} hold.
  Define the transfer functions $G(s) = H(s) + H(s)\herm$ and
  $\hG(s) = \hH(s) + \hH(s)\herm$.
  Then
  \begin{subequations}\label{eq:phG}
    \begin{align}
      \label{eq:phG1}
      G\myparen*{-\overline{\lambda_i}}
      b_i
       & =
      \hG\myparen*{-\overline{\lambda_i}}
      b_i,                                 \\
      \label{eq:phG2}
      b_i\herm
      G\myparen*{-\overline{\lambda_i}}
       & =
      b_i\herm
      \hG\myparen*{-\overline{\lambda_i}}, \\
      \label{eq:phG3}
      b_i\herm
      \frac{\partial G}{\partial s}\myparen*{-\overline{\lambda_i}}
      b_i
       & =
      b_i\herm
      \frac{\partial \hG}{\partial s}\myparen*{-\overline{\lambda_i}}
      b_i,
    \end{align}
  \end{subequations}
  for $i = 1, 2, \dots, \nrom$.
\end{corollary}
\begin{proof}
  The right tangential Lagrange conditions~\eqref{eq:phG1} follows directly
  from \Cref{thm:pH-interp} and from the definitions of $G$ and $\hG$.
  Due to symmetry, we also get the left tangential Lagrange
  conditions~\eqref{eq:phG2}.
  The bitangential Hermite condition~\eqref{eq:phG3} also follows directly by
  using Wirtinger calculus.
\end{proof}
Thus once again, bitangential Hermite interpolation
(of a modified transfer function)
appears as the necessary condition for $\Htwo$-optimality.

\subsection{Numerical Example}

To demonstrate the results of \Cref{thm:pH-interp},
we choose the \ac{fom} as
\begin{gather*}
  J =
  \begin{bmatrix}
    0 & 0  & 0 & 0 & 0  & 0 \\
    0 & 0  & 1 & 0 & 0  & 0 \\
    0 & -1 & 0 & 0 & 0  & 0 \\
    0 & 0  & 0 & 0 & 0  & 0 \\
    0 & 0  & 0 & 0 & 0  & 1 \\
    0 & 0  & 0 & 0 & -1 & 0
  \end{bmatrix}\!, \
  R =
  \mydiag*{\frac{1}{10}, \frac{1}{10}, \frac{1}{10}, 1, 1, 1}, \
  B =
  \begin{bmatrix}
    1 & 1 \\
    1 & 2 \\
    1 & 3 \\
    1 & 4 \\
    1 & 5 \\
    1 & 6
  \end{bmatrix}\!,
\end{gather*}
which has three (one real, one complex conjugate pair) poles
close to the imaginary axis and three poles that are further away:
\[
  -\frac{1}{10}, \
  -\frac{1}{10} \pm \imag, \
  -1, \
  -1 \pm \imag.
\]
We then find \iac{diagstrom} of order $\nrom = 3$ using Nelder-Mead initialized
with
\[
  \hJ = I_{6, 3}\tran J I_{6, 3}, \
  \hR = I_{6, 3}\tran R I_{6, 3}, \
  \hB = I_{6, 3}\tran B,
\]
where $I_{6, 3} \in \RR^{6 \times 3}$ are the first three columns of the
$6 \times 6$ identity matrix.
The relative $\Htwo$ error of the initial \ac{ph} \ac{rom} is $0.74487$.
The \ac{ph} \ac{rom} in the optimization process is parameterized using a vector
$x \in \RR^{9}$ as
\begin{gather*}
  \hJ =
  \begin{bmatrix}
    0 & 0    & 0   \\
    0 & 0    & x_1 \\
    0 & -x_1 & 0
  \end{bmatrix}\!, \
  \hR = \mydiag{x_2, x_3, x_3}, \
  \hB =
  \begin{bmatrix}
    x_4 & x_5 \\
    x_6 & x_7 \\
    x_8 & x_9
  \end{bmatrix}\!.
\end{gather*}
Note that every real \ac{ph} \ac{lti} system $(\hJ, \hR, \hB)$ of order~$3$
such that $\hJ - \hR$ is normal and with one real and one complex conjugate
pair of poles
is of the above form after a state space transformation
$(\hT\tran \hJ \hT, \hT\tran \hR \hT, \hT\tran \hB)$
by an orthogonal matrix $\hT$.
The result of the optimization is \iac{ph} \ac{rom} with
(rounded to $5$ significant digits)
\begin{gather*}
  \hJ - \hR =
  \begin{bmatrix}
    -0.25493 & 0        & 0        \\
    0        & -0.47415 & 0.99652  \\
    0        & -0.99652 & -0.47415
  \end{bmatrix}\!, \
  \hB =
  \begin{bmatrix}
    -0.92106 & -2.7957 \\
    1.5589   & 5.4235  \\
    1.4045   & 6.4414
  \end{bmatrix}\!,
\end{gather*}
leading to a relative $\Htwo$ error of $0.19595$.
Its poles are
\[
  \lambda_1 = -0.25493, \
  \lambda_{2, 3} = -0.47415 \pm 0.99652 \imag.
\]
Then we check the optimality conditions and find that
\begin{gather*}
  \mathrm{reldist}\myparen*{
    \myparen*{
      H\myparen*{-\overline{\lambda_1}}
      + H\myparen*{-\overline{\lambda_1}}\herm
    }
    b_1,
    \myparen*{
      \hH\myparen*{-\overline{\lambda_1}}
      + \hH\myparen*{-\overline{\lambda_1}}\herm
    }
    b_1
  }
  \approx 2.9 \times 10^{-9}, \\
  \mathrm{reldist}\myparen*{
    \myparen*{
      H\myparen*{-\overline{\lambda_{2, 3}}}
      + H\myparen*{-\overline{\lambda_{2, 3}}}\herm
    }
    b_{2, 3},
    \myparen*{
      \hH\myparen*{-\overline{\lambda_{2, 3}}}
      + \hH\myparen*{-\overline{\lambda_{2, 3}}}\herm
    }
    b_{2, 3}
  }
  \approx 5.8 \times 10^{-9}, \\
  \mathrm{reldist}\myparen*{
    b_1\herm
    H'\myparen*{-\overline{\lambda_1}}
    b_1,
    b_1\herm
    \hH'\myparen*{-\overline{\lambda_1}}
    b_1
  }
  \approx 2.5 \times 10^{-8}, \\
  \mathrm{reldist}\myparen*{
    b_{2, 3}\herm
    H'\myparen*{-\overline{\lambda_{2, 3}}}
    b_{2, 3},
    b_{2, 3}\herm
    \hH'\myparen*{-\overline{\lambda_{2, 3}}}
    b_{2, 3}
  }
  \approx 2.5 \times 10^{-8},
\end{gather*}
which is aligned with \Cref{thm:pH-interp}.
We also find that
\begin{gather*}
  \mathrm{reldist}\myparen*{
    H\myparen*{-\overline{\lambda_1}}
    b_1,
    \hH\myparen*{-\overline{\lambda_1}}
    b_1
  }
  \approx 5.8 \times 10^{-3}, \\
  \mathrm{reldist}\myparen*{
    H\myparen*{-\overline{\lambda_{2, 3}}}
    b_{2, 3},
    \hH\myparen*{-\overline{\lambda_{2, 3}}}
    b_{2, 3}
  }
  \approx 2.1 \times 10^{-2},
\end{gather*}
which demonstrates that the \ac{diagstrom} does not satisfy the unstructured
interpolatory necessary optimality conditions~\eqref{eq:lti-oc}.
The code is available at~\cite{Mli24}.

\section{Time-delay Systems}%
\label{sec:td}
Beyond standard \ac{lti} systems~\eqref{eq:lti-fom},
dynamical systems with internal and/or input/output delays
appear in many control systems (see, e.g.,~\cite{Fri14}).
Here, we consider a time-delay \ac{strom} of the form
\begin{subequations}\label{eq:td-rom}
  \begin{align}
    \dot{\hx}(t) & = \hA \hx(t) + \hA_{\tau} \hx(t - \tau) + \hB u(t), \\*
    \hy(t)       & = \hC \hx(t),
  \end{align}
\end{subequations}
where $\hA_{\tau} \in \RRrr$ and $\tau > 0$ is the internal delay.
We use $(\hA, \hA_{\tau}, \hB, \hC)$ to denote~\eqref{eq:td-rom}.
The system can contain more than one delay,
but we focus on the single-delay case for simplicity.
The transfer function of the time-delay \ac{strom}~\eqref{eq:td-rom} is given by
\begin{equation} \label{eq:delaytf}
  \hH(s) = \hC \myparen*{s \hI - \hA - e^{-\tau s} \hA_{\tau}}^{-1} \hB.
\end{equation}
We assume that $\lambda \hI - \hA - e^{-\tau \lambda} \hA_{\tau}$ is
asymptotically stable.
The goal is the same as before:
find an optimal reduced time-delay system as in~\eqref{eq:delaytf} that
minimizes the $\Htwo$ distance between the original transfer function $H$ and
the reduced $\hH$.

The works~\cite{SinGB16,AumW23} proposed \ac{irka}-type algorithms to construct
\acp{strom}, including time-delay systems.
Even though both algorithms work well in practice,
they do not guarantee $\Htwo$-optimality and
are not based on true $\Htwo$-optimality conditions.
An $\Htwo$-optimal method for time-delay system is proposed in~\cite{GomEMM19},
but this approach uses gradient-based optimization
based on the system Gramians and
does not reveal or consider optimal interpolatory conditions.
That is precisely what we establish in this section.

\subsection{Interpolatory Conditions}

For SISO systems and for the simple case of $\nrom = 1$ and $\hA = 0$
in~\eqref{eq:td-rom},
the interpolatory necessary $\Htwo$-optimality conditions were
derived in~\cite{PonGB+16}.
The approach in~\cite{PonGB+16} can be extended to diagonal MIMO time-delay
systems for an arbitrary reduced order $\nrom \ge 1$~\cite{Pon23}.
Below we derive the more general interpolatory $\Htwo$-optimality conditions,
but using our generalized optimality framework from \Cref{thm:cond-diag-s}.

For the time-delay system~\eqref{eq:td-rom},
the diagonal structure that we assume for \acp{rom}
corresponds to assuming $\hA$ and $\hA_{\tau}$ are
simultaneously diagonalizable.
Therefore, there exists an invertible matrix $\hT \in \CCrr$ such that
$\hT^{-1} \hA \hT = \hM$, and
$\hT^{-1} \hA_{\tau} \hT = \hSigma$, with
$\hM = \mydiag{\mu_1, \mu_2, \ldots, \mu_{\nrom}}$ and
$\hSigma = \mydiag{\sigma_1, \sigma_2, \ldots, \sigma_{\nrom}}$.
Then, the transfer function $\hH$ in~\eqref{eq:delaytf}
can be equivalently rewritten as
\begin{align}
  \nonumber
  \hH(s)
   & =
  \hC
  \myparen*{
    s \hT \hT^{-1}
    - \hT \hM \hT^{-1}
    - e^{-\tau s} \hT \hSigma \hT^{-1}
  }^{-1}
  \hB  \\
  \nonumber
   & =
  \hC
  \hT
  \myparen*{s \hI - \hM - e^{-\tau s} \hSigma}^{-1}
  \hT^{-1}
  \hB  \\
  \label{eq:td-pole-res-form}
   & =
  \sum_{i = 1}^{\nrom}
  \frac{c_i b_i\herm}{s - \mu_i - e^{-\tau s} \sigma_i},
\end{align}
where $c_i = \hC \hT e_i$ and $b_i = \hB\tran \hS e_i$.
To find the zeros of $s - \mu_i - e^{-\tau s} \sigma_i$,
we can use the Lambert~$W$ function
\begin{equation}\label{eq:td-poles}
  \lambda_{ij}
  =
  \mu_i
  + \frac{1}{\tau} W_j\myparen*{\tau \sigma_i e^{-\tau \mu_i}},
\end{equation}
where $\lambda_{ij}$ is the pole corresponding to the $j$th branch of the
Lambert~$W$ function (see~\cite{CepM15}).

Note that the reformulation of $\hH$ as in~\eqref{eq:td-pole-res-form}
perfectly aligns with the form of transfer functions with general denominators
in \Cref{thm:cond-diag-s}.
The next result uses this observation and applies \Cref{thm:cond-diag-s}
to the transfer function~\eqref{eq:td-pole-res-form}
to develop the interpolatory necessary $\Htwo$-optimality conditions
for MIMO time-delay systems.
\begin{theorem}\label{thm:td-cond}
  Let $H \in \Htwo^{\nout \times \nin}$, $\tau > 0$, and
  $\Sigma_{\textnormal{D-TD}}$ be the set of time-delay \acp{diagstrom}
  $(\hM, \hSigma, \hB, \hC)$ such that
  $\hM = \mydiag{\mu_i}, \hSigma = \mydiag{\sigma_i} \in \CCrr$ are diagonal,
  $\lambda \hI - \hM - e^{-\tau \lambda} \hSigma$ has all of its eigenvalues in
  $\CC_-$,
  $\hB \in \CCri$, and
  $\hC \in \CCor$.
  Let $(\hM, \hSigma, \hB, \hC)$ be an $\Htwo$-optimal time-delay \ac{diagstrom}
  for $H$ in $\Sigma_{\textnormal{D-TD}}$.
  If $\sigma_i \ne 0$ and
  $\hH$ in~\eqref{eq:td-pole-res-form} has pairwise distinct simple poles
  $\lambda_{ij}$ as in~\eqref{eq:td-poles},
  then
  \begin{subequations}\label{eq:td-cond}
    \begin{gather}
      \displaybreak[1]
      \label{eq:td-cond1}
      \myparen*{
        \sum_{j = -\infty}^{\infty}
        \overline{\phi_{ij}}
        H\myparen*{-\overline{\lambda_{ij}}}
      }
      b_i
      =
      \myparen*{
        \sum_{j = -\infty}^{\infty}
        \overline{\phi_{ij}}
        \hH\myparen*{-\overline{\lambda_{ij}}}
      }
      b_i, \\
      \displaybreak[1]
      \label{eq:td-cond2}
      c_i\herm
      \myparen*{
        \sum_{j = -\infty}^{\infty}
        \overline{\phi_{ij}}
        H\myparen*{-\overline{\lambda_{ij}}}
      }
      =
      c_i\herm
      \myparen*{
        \sum_{j = -\infty}^{\infty}
        \overline{\phi_{ij}}
        \hH\myparen*{-\overline{\lambda_{ij}}}
      }, \\
      \displaybreak[1]
      \label{eq:td-cond3}
      \begin{aligned}
         &
        c_i\herm
        \myparen*{
          \sum_{j = -\infty}^{\infty}
          \myparen*{
            \overline{\psi_{ij}}
            H'\myparen*{-\overline{\lambda_{ij}}}
            -
            \overline{\rho_{ij}}
            H\myparen*{-\overline{\lambda_{ij}}}
          }
        }
        b_i  \\*
         & =
        c_i\herm
        \myparen*{
          \sum_{j = -\infty}^{\infty}
          \myparen*{
            \overline{\psi_{ij}}
            \hH'\myparen*{-\overline{\lambda_{ij}}}
            -
            \overline{\rho_{ij}}
            \hH\myparen*{-\overline{\lambda_{ij}}}
          }
        }
        b_i, \text{ and}
      \end{aligned} \\
      \label{eq:td-cond4}
      \begin{aligned}
         &
        c_i\herm
        \myparen*{
          \sum_{j = -\infty}^{\infty}
          \myparen*{
            \myparen*{
              \overline{\phi_{ij}}
              - \overline{\psi_{ij}}
            }
            H'\myparen*{-\overline{\lambda_{ij}}}
            +
            \overline{\rho_{ij}}
            H\myparen*{-\overline{\lambda_{ij}}}
          }
        }
        b_i  \\*
         & =
        c_i\herm
        \myparen*{
          \sum_{j = -\infty}^{\infty}
          \myparen*{
            \myparen*{
              \overline{\phi_{ij}}
              - \overline{\psi_{ij}}
            }
            \hH'\myparen*{-\overline{\lambda_{ij}}}
            +
            \overline{\rho_{ij}}
            \hH\myparen*{-\overline{\lambda_{ij}}}
          }
        }
        b_i,
      \end{aligned}
    \end{gather}
  \end{subequations}
  for $i = 1, 2, \ldots, \nrom$,
  where
  \[
    \phi_{ij} = \frac{1}{1 + \tau (\lambda_{ij} - \mu_i)},\
    \psi_{ij} = \frac{1}{{(1 + \tau (\lambda_{ij} - \mu_i))}^2}, \text{ and }
    \rho_{ij} = \frac{\tau^2 (\lambda_{ij} - \mu_i)}
    {{(1 + \tau (\lambda_{ij} - \mu_i))}^3}.
  \]
\end{theorem}
\begin{proof}
  In the setting of \Cref{thm:cond-diag-s}, we have
  $\car_1(s) = s$,
  $\car_2(s) = -1$,
  $\car_3(s) = -e^{-\tau s}$, and
  $a_i(s) = s - \mu_i - e^{-\tau s} \sigma_i$.
  Existence of the admissible contours needed in \Cref{thm:cond-diag-s} is
  guaranteed by~\cite[Theorem~12.13]{BelC63} and~\cite[Theorem~2.2]{ZwaCPG88}.
  Note that
  \begin{subequations}
    \begin{align}
      \label{eq:td-a-lambda}
      a_i(\lambda_{ij})
       & = \lambda_{ij} - \mu_i - e^{-\tau \lambda_{ij}} \sigma_i
      = 0,                                                        \\
      \label{eq:td-a-prime-lambda}
      a_i'(\lambda_{ij})
       & = 1 + \tau e^{-\tau \lambda_{ij}} \sigma_i
      = 1 + \tau (\lambda_{ij} - \mu_i),                          \\
      \label{eq:td-a-prime-prime-lambda}
      a_i''(\lambda_{ij})
       & = -\tau^2 e^{-\tau \lambda_{ij}} \sigma_i
      = -\tau^2 (\lambda_{ij} - \mu_i),                           \\
      \label{eq:td-alpha3-lambda}
      \car_3(\lambda_{ij})
       & = -e^{-\tau \lambda_{ij}}
      = -\frac{\lambda_{ij} - \mu_i}{\sigma_i},                   \\
      \label{eq:td-alpha3-prime-lambda}
      \car_3'(\lambda_{ij})
       & = \tau e^{-\tau \lambda_{ij}}
      = \frac{\tau (\lambda_{ij} - \mu_i)}{\sigma_i},
    \end{align}
  \end{subequations}
  where we used~\eqref{eq:td-a-lambda} in the later expressions.
  Using~\eqref{eq:td-a-prime-lambda},~\eqref{eq:td-a-prime-prime-lambda},
  $\car_2(\lambda_{ij}) = -1$, and $\car_2'(\lambda_{ij}) = 0$,
  we find
  \begin{equation}\label{eq:td-res1}
    \frac{1}{a_i'(\lambda_{ij})} = \phi_{ij},  \quad
    \frac{\car_2(\lambda_{ij})}{{a_i'(\lambda_{ij})}^2} = -\psi_{ij}, \quad
    \frac{\car_2'(\lambda_{ij})}{{a_i'(\lambda_{ij})}^2}
    - \frac{\car_2(\lambda_{ij}) a_i''(\lambda_{ij})}{{a_i'(\lambda_{ij})}^3}
    = -\rho_{ij}.
  \end{equation}
  Next, together with
  $\phi_{ij} - \psi_{ij}
    = \frac{\tau (\lambda_{ij} - \mu_i)}{{a_i'(\lambda_{ij})}^2}$,
  expression~\eqref{eq:td-alpha3-lambda} gives
  \begin{equation}\label{eq:td-cond4-coeff1}
    \frac{\car_3(\lambda_{ij})}{{a_i'(\lambda_{ij})}^2}
    = -\frac{\lambda_{ij} - \mu_i}{\sigma_i {a_i'(\lambda_{ij})}^2}
    = -\frac{1}{\tau \sigma_i} (\phi_{ij} - \psi_{ij}).
  \end{equation}
  Then, using~\eqref{eq:td-a-prime-lambda}, \eqref{eq:td-a-prime-prime-lambda},
  \eqref{eq:td-alpha3-lambda}, and~\eqref{eq:td-alpha3-prime-lambda} leads to
  \begin{align}
    \nonumber
    \frac{
      \car_3'(\lambda_{ij})
    }{
      {a_i'(\lambda_{ij})}^2
    }
    - \frac{
      \car_3(\lambda_{ij})
      a_i''(\lambda_{ij})
    }{
      {a_i'(\lambda_{ij})}^3
    }
     & =
    \frac{
      \tau
      (\lambda_{ij} - \mu_i)
    }{
      \sigma_i
      {(1 + \tau (\lambda_{ij} - \mu_i))}^2
    }
    -
    \frac{
    \tau^2
    {(\lambda_{ij} - \mu_i)}^2
    }{
    \sigma_i
    {(1 + \tau (\lambda_{ij} - \mu_i))}^3
    }    \\
    \nonumber
     & =
    \frac{\tau}{\sigma_i}
    \cdot
    \frac{
      (\lambda_{ij} - \mu_i)
      (1 + \tau (\lambda_{ij} - \mu_i))
      -
      \tau
      {(\lambda_{ij} - \mu_i)}^2
    }{
      {(1 + \tau (\lambda_{ij} - \mu_i))}^3
    }    \\
    \nonumber
     & =
    \frac{\tau}{\sigma_i}
    \cdot
    \frac{
      \lambda_{ij} - \mu_i
    }{
      {(1 + \tau (\lambda_{ij} - \mu_i))}^3
    }    \\
    \label{eq:td-cond4-coeff2}
     & =
    \frac{1}{\tau \sigma_i}
    \rho_{ij}.
  \end{align}
  Using~\eqref{eq:td-res1},
  conditions~\eqref{eq:cond-diag-s-1} and~\eqref{eq:cond-diag-s-2}
  yield~\eqref{eq:td-cond1} and~\eqref{eq:td-cond2}, respectively.
  Furthermore, the condition~\eqref{eq:cond-diag-s-3} with $\car_2$
  implies~\eqref{eq:td-cond3}.
  Inserting~\eqref{eq:td-cond4-coeff1}
  and~\eqref{eq:td-cond4-coeff2} into the condition~\eqref{eq:cond-diag-s-3}
  with $\car_3$ implies~\eqref{eq:td-cond4}.
\end{proof}
Not surprisingly, the interpolatory optimality conditions for time-delay systems
appear much more involved than the earlier ones for second-order systems and
\ac{ph} systems.
The single interpolation conditions are replaced by an (infinite) weighted sum,
reflecting the fact that time-delay systems have infinitely many poles
(identified via the Lambert W function).
Moreover, the Hermite conditions in this case appear as mixed conditions
where a linear combination of $\hH$ and $\hH'$ needs to be interpolated.
However, one aspect stays the same:
interpolation needs to happen at the mirror images of the poles.

Note that we can replace the condition~\eqref{eq:td-cond4} by
the addition of~\eqref{eq:td-cond4} and~\eqref{eq:td-cond3},
which simply becomes
\begin{equation*}
  c_i\herm
  \myparen*{
    \sum_{j = -\infty}^{\infty}
    \overline{\phi_{ij}}
    H'\myparen*{-\overline{\lambda_{ij}}}
  }
  b_i
  =
  c_i\herm
  \myparen*{
    \sum_{j = -\infty}^{\infty}
    \overline{\phi_{ij}}
    \hH'\myparen*{-\overline{\lambda_{ij}}}
  }
  b_i.
\end{equation*}

\subsection{Numerical Example}

We demonstrate the results of \Cref{thm:td-cond} on the \ac{fom} with
\begin{gather*}
  A =
  \begin{bmatrix}
    -\frac{1}{10} & 0             & 0             \\
    0             & -\frac{1}{10} & 1             \\
    0             & -1            & -\frac{1}{10}
  \end{bmatrix}\!, \
  A_{\tau} =
  \begin{bmatrix}
    -\frac{1}{10} & 0             & 0             \\
    0             & -\frac{1}{10} & \frac{1}{10}  \\
    0             & -\frac{1}{10} & -\frac{1}{10}
  \end{bmatrix}\!, \\
  B =
  \begin{bmatrix}
    1 & 1 & 1 \\
    1 & 2 & 2 \\
    1 & 3 & 1
  \end{bmatrix}\!, \
  C =
  \begin{bmatrix}
    1 & 1 & 1 \\
    1 & 2 & 3
  \end{bmatrix}\!,
\end{gather*}
where $A$ was chosen to have three eigenvalues close to the imaginary axis
(one real and one complex pair),
while $A_{\tau}$ has a similar structure and is sufficiently small to ensure
that the poles of the \ac{fom} have negative real parts.
We then find \iac{diagstrom} of order $\nrom = 2$ using Nelder-Mead initialized
with
\[
  \hA = I_{3, 2}\tran A I_{3, 2}, \
  \hA_{\tau} = I_{3, 2}\tran A_{\tau} I_{3, 2}, \
  \hB = I_{3, 2}\tran B, \
  \hC = C I_{3, 2},
\]
where $I_{3, 2} \in \RR^{3 \times 2}$ are the first two columns of the
$3 \times 3$ identity matrix.
The relative $\Htwo$ error of the initial \ac{diagstrom} is $1.0237$
(computed using numerical quadrature).
The \ac{diagstrom} in Nelder-Mead is parameterized using a vector
$x \in \RR^{14}$ as
\begin{gather*}
  \hA =
  \begin{bmatrix}
    x_1 & 0   \\
    0   & x_2
  \end{bmatrix}\!, \
  \hA_{\tau} =
  \begin{bmatrix}
    x_3 & 0   \\
    0   & x_4
  \end{bmatrix}\!, \
  \hB =
  \begin{bmatrix}
    x_5 & x_6 & x_7    \\
    x_8 & x_9 & x_{10}
  \end{bmatrix}\!, \
  \hC =
  \begin{bmatrix}
    x_{11} & x_{12} \\
    x_{13} & x_{14}
  \end{bmatrix}\!.
\end{gather*}
The resulting optimal model is a \ac{diagstrom}
(rounded to $5$ significant digits)
\begin{gather*}
  \hA =
  \begin{bmatrix}
    0.44909 & 0       \\
    0       & 0.17539
  \end{bmatrix}\!, \
  \hA_{\tau} =
  \begin{bmatrix}
    -1.2372 & 0       \\
    0       & -1.2408
  \end{bmatrix}\!, \\
  \hB =
  \begin{bmatrix}
    4.3506  & 12.827  & 4.2756  \\
    -11.793 & -26.685 & -17.774
  \end{bmatrix}\!, \
  \hC =
  \begin{bmatrix}
    0.15039 & -0.08129 \\
    0.32807 & -0.19743
  \end{bmatrix}\!,
\end{gather*}
with relative $\Htwo$ error of $0.589$.
Its $8$ poles with largest real parts are
\begin{equation*}
  -0.03289 \pm 1.1843 \imag, \
  -0.11458 \pm 1.3609 \imag, \
  -1.8458 \pm 7.5938 \imag, \
  -1.8541 \pm 7.5582 \imag.
\end{equation*}
Then we find that the relative errors in~\eqref{eq:td-cond} are between
$1.8 \times 10^{-9}$ and $8.4 \times 10^{-8}$
(we truncate the double sums from $-10^5$ to $10^5$),
thus numerically validating the optimality conditions.
The code is available at~\cite{Mli24}.

\section{Conclusion}%
\label{sec:conclusion}
We have developed interpolatory $\Htwo$-optimality
conditions for approximating non-parametric structured dynamical systems, namely
for second-order, \ac{ph}, and time-delay systems.
We have shown that bitangential Hermite interpolation is the common unifying
framework across all these different settings.
In this paper, we have mainly focused on the theoretical analysis of deriving
the optimality conditions.
A natural future direction would be to develop (iterative) numerical methods,
such as \ac{irka}, that directly use the interpolation conditions discussed here
as opposed to the gradient-based iterative optimization methods employed in the
literature.



\bibliographystyle{alphaurl}
\bibliography{my}
\addcontentsline{toc}{chapter}{References}
\end{document}